\documentclass[11pt]{article}

\usepackage[utf8]{inputenc}

\usepackage{geometry}
\usepackage{graphicx}

\usepackage{booktabs} 
\usepackage{array} 
\usepackage{paralist} 
\usepackage{verbatim} 
\usepackage{subfig} 
\usepackage{tabularx}
\usepackage{fancyhdr} 
\usepackage[nottoc,notlof,notlot]{tocbibind}
\usepackage[titles,subfigure]{tocloft}

\usepackage{amsmath,amsfonts,amsthm,mathrsfs,amssymb,cite}
\usepackage[usenames]{color}

\newcommand{\R}{{\mathbb R}}

\newcommand{\N}{{\mathbb N}}
\newcommand{\C}{{\mathbb C}}
\newcommand{\s}{{\mathbb S}}

\newcommand{\no}{\nonumber}
\newcommand{\be}{\begin{eqnarray}}
\newcommand{\ben}{\begin{eqnarray*}}
\newcommand{\en}{\end{eqnarray}}
\newcommand{\enn}{\end{eqnarray*}}

\newcommand{\curl}{{\rm curl\,}}

\newcommand{\grad}{{\rm grad\,}}
\newcommand{\divv}{{\rm div\,}}

\newcommand{\Om}{\Omega}

\newcommand{\la}{\lambda}

\newtheorem{thm}{Theorem}[section]

\newtheorem{lem}{Lemma}[section]

\theoremstyle{definition}
\newtheorem{defn}{Definition}[section]
\theoremstyle{remark}

\newtheorem{rem}{Remark}[section]
\numberwithin{equation}{section}

\title{\bf Inverse Elastic Scattering for Multiscale Rigid Bodies with A Single Far-field Pattern}

\author{Guanghui Hu \thanks{Weierstrass Institute, Mohrenstr. 39,
10117 Berlin,
Germany. Email: {\tt hu@wias-berlin.de}},\quad
Jingzhi Li\thanks{Faculty of Science, South University of Science and
Technology of China, 518055 Shenzhen, P.~R.~China. Email: {\tt li.jz@sustc.edu.cn}},\quad Hongyu Liu\thanks{Department of Mathematics and Statistics, University of North Carolina, Charlotte, NC 28223, USA.   Email:  {\tt hongyu.liuip@gmail.com}},\quad  Hongpeng Sun\thanks{Institute for Mathematics and Scientific Computing, University of Graz, Heinrichstr. 36, A-8010 Graz, Austria. Email: {\tt hongpeng.sun@uni-graz.at} } }

\date{} 
\begin{document}
\maketitle

\begin{abstract}
We develop three inverse elastic scattering schemes for locating multiple small, extended and multiscale rigid bodies, respectively. There are some salient and promising features of the proposed methods. The cores of those schemes are certain indicator functions, which are obtained by using only a single far-field pattern of the pressure (longitudinal) wave, or the shear (transversal) wave, or the total wave field. Though the inverse scattering problem is known to be nonlinear and ill-posed, the proposed reconstruction methods are totally ``direct" and there are no inversions involved. Hence, the methods are very efficient and robust against noisy data. Both rigorous mathematical justifications and numerical simulations are presented in our study.
\end{abstract}

\section{Introduction}
\setcounter{equation}{0}
The elastic wave propagation problems have a wide range of applications,  particularly in geophysics, nondestructive testing and seismology.
The associated inverse problems arise from the use of transient elastic waves to identify the elastic properties as well as to detect flaws and cracks of solid speciments, especially in the nondestructive evaluation of concrete structures (see e.g. \cite{application2,application3}).  Moreover, the problem of elastic pulse transmission and reflection through the earth is fundamental to both the investigation of earthquakes and
the utility of seismic waves in search for oil and ore bodies (see, e.g., \cite{Abubakar,Fokkema,Fokkema2,LE,Sherwood}
and the references therein).
The scattering of elastic waves are very complicated due to the coexistence of compressional and shear waves propagating at different speeds.
For a rigid elastic body, these two waves are coupled at the scattering surface and the total displacement field vanishes there. In this paper, we are concerned with the inverse problem of identifying a collection of unknown rigid elastic scatterers by using the corresponding far-field measurement. In what follows, we first present the mathematical formulations of the direct and inverse elastic scattering problems for our study.


Consider a time-harmonic elastic plane  wave $u^{in}(x)$, $x\in\mathbb{R}^3$ (with the time variation of the form $e^{-i\omega t}$ being factorized out, where $\omega\in\mathbb{R}_+$ denotes the angular frequency) impinged on a rigid scatterer $D\subset\R^3$ embedded in an infinite isotropic and homogeneous elastic medium in $\R^3$. The elastic scattering is governed by the reduced Navier equation (or Lam\'e system)
\be\label{Navier}
 (\Delta^*+\omega^2)\,u=0,\quad \mbox{in}\quad \R^3\backslash \overline{D}, \quad\Delta^*:=\mu\Delta+(\la+\mu)\,\grad\divv\,
\en
where $u(x)\in\mathbb{C}^3$ denotes the total displacement field, and $\la$,  $\mu$ are the Lam\'e constants satisfying $\mu>0$ and $3\la+2\mu>0$. Here, we note that the density of the background elastic medium has been normalized to be unitary. Henceforth, we suppose that $D\subset\R^3$ is an open bounded domain such that $\R^3\backslash\overline{D}$ is connected. It is emphasized that $D$ may consist of (finitely many) multiple simply connected components.
The incident elastic plane wave is of the following general form
\be\label{plane}
u^{in}(x)=u^{in}(x; d, d^\perp,\alpha,\beta,\omega)= \alpha\,de^{ik_px\cdot d}+\beta d^\perp e^{ik_sx\cdot d},\quad \alpha,\beta\in \C,
\en
where $d\in\mathbb{S}^2:=\{x\in\R^3: |x|=1\},$ is the impinging direction, $d^\bot\in\mathbb{S}^2$ satisfying $d^\bot\cdot d=0$ is the polarization direction; and $k_s:=\omega/\sqrt{\mu}$, $k_p:=\omega/\sqrt{\la+2\mu}$ denote the shear and compressional wave numbers, respectively. If $\alpha=1, \beta=0$ for $u^{in}$ in \eqref{plane}, then $u^{in}=u_p^{in}:=d e^{ik_px\cdot d}$ is the (normalized) plane pressure wave; whereas if $\alpha=0, \beta=1$ for $u^{in}$ in \eqref{plane}, then $u^{in}=u_s^{in}:=d^\perp e^{ik_sx\cdot d}$ is the (normalized) plane shear wave. The obstacle $D$ is a rigid body and $u$ satisfies the first kind (Dirichlet)  boundary condition
\be\label{D}
u=0\quad \mbox{on}\quad\partial D.
\en
Define $u^{sc}:=u-u^{in}$ to be the {\it scattered wave}, which can be easily verified to satisfy the Navier equation (\ref{Navier}) as well.    $u^{sc}$ can be decomposed into the sum
\ben
u^{sc}:=u_p^{sc}+u_s^{sc},\quad u_p^{sc}:=-\frac{1}{k_p^2}\grad\divv u^{sc},\quad u_s^{sc}:=\frac{1}{k_s^2}\curl\curl u^{sc},
\enn
where the vector functions $u_p^{sc}$ and $u_s^{sc}$ are referred to as the pressure (longitudinal) and shear (transversal) parts of $u^{sc}$, respectively, satisfying
\ben
&&(\Delta + k_p^2 )\, u_p^{sc}=0,\quad \curl u_p^{sc}=0,\quad\mbox{in}\;\R^3\backslash\overline{D},\\
&&(\Delta + k_s^2 )\, u_s^{sc}=0,\quad \divv u_s^{sc}=0,\quad\ \mbox{in}\;\R^3\backslash\overline{D}.
\enn Moreover, the scattered field $u^{sc}$ is required to satisfy  Kupradze's radiation condition
\be\label{Rads}
\lim_{r\rightarrow\infty}\left(\frac{\partial u_p^{sc}}{\partial r}-ik_p u_p^{sc}\right)=0,\quad \lim_{r\rightarrow\infty}\left(\frac{\partial u_s^{sc}}{\partial r}-ik_s u_s^{sc}\right)=0,\quad r=|x|,
\en uniformly in all directions $\hat{x}=x/|x|\in \mathbb{S}^2$  (see, e.g., \cite{AK}). The radiation conditions in (\ref{Rads}) lead to
the P-part (longitudinal part) $u^{\infty}_p$ and the S-part (transversal part) $u^{\infty}_s$ of the far-field pattern of $u^{sc}$, read off from the large $|x|$ asymptotics (after normalization)
\be\label{far}
u^{sc}(x)=\frac{\exp(ik_p|x|)}{4\pi(\lambda+\mu)|x|}\, u^{\infty}_p(\hat{x})+ \frac{\exp(ik_s|x|)}{4\pi \mu |x|}\, u^{\infty}_s(\hat{x})+ \mathcal{O}(\frac{1}{|x|^2}),\quad|x|\rightarrow+\infty.
\en
Note that $u^{\infty}_p(\hat{x})$ and $u^{\infty}_s(\hat{x})$ are also known as the far-field patterns of $u^{sc}_p$ and $u^{sc}_s$, respectively.
In this paper, we define the full far-field pattern $u^\infty$ of the scattered field $u^{sc}$ as the sum of $u^{\infty}_p$ and $u^{\infty}_s$, i.e.,
\be\label{eq:full fp}
 u^\infty(\hat{x}):=u^{\infty}_p(\hat{x})+u^{\infty}_s(\hat{x}).
\en
Since $u^{\infty}_p(\hat{x})$
 is normal to $\mathbb{S}^2$ and $u^{\infty}_s(\hat{x})$ is  tangential to $\mathbb{S}^2$, there holds
\ben u^{\infty}_p(\hat{x})=(u^\infty(\hat{x})\cdot \hat{x})\,\hat{x}, \quad u^{\infty}_s(\hat{x})=\hat{x}\times u^\infty(\hat{x})\times \hat{x},\quad \hat x\in\mathbb{S}^2.
\enn

 The direct elastic scattering problem (DP) is stated as follows
\begin{description}
\item[(\textbf{DP}):] Given a rigid scatterer $D\subset\R^3$ and an incident plane wave $u^{in}$ of the form (\ref{plane}), find the total field $u=u^{in}+u^{sc}$
in $\R^3\backslash \overline{D}$ such that the Dirichlet boundary condition (\ref{D}) holds on $\partial D$ and that
the scattered field $u^{sc}$ satisfies Kupradze's radiation condition (\ref{Rads}).
\end{description}
We refer to the monograph \cite{Kupradze} for a comprehensive treatment of the boundary value problems of elasticity.
It is well-known that  the direct scattering problem admits a unique solution $u\in C^2(\R^3\backslash \overline{D})^3 \cap C^1(\R^3\backslash D)^3$
if $\partial D$ is $C^2$-smooth (see \cite{Kupradze}), whereas
 $u\in H^1_{loc}(\R^3\backslash \overline{D})^3$ if $\partial D$ is Lipschitz (see \cite{BP2008}).

Throughout the rest of the paper, $u_\tau^{\infty}(\hat{x})$ with $\tau=\emptyset$ signifies the full far-field pattern defined in \eqref{eq:full fp}. In this paper,
we are interested in  the following inverse problem (IP),
\begin{description}
\item[\bf (IP):] Recover the rigid scatterer $D$ from knowledge (i.e., the measurement) of the
far-field pattern $u_\tau^{\infty}(\hat{x}; d, d^\perp, \alpha,\beta, \omega)$ ($\tau=p, s\;{\rm or}\; \emptyset$).
\end{description}
Note that in (IP), the measurement data can be the P-part far-field pattern $u_p^{\infty}$,  the S-part far-field pattern
$u_s^{\infty}$, or the full far-field pattern $u^{\infty}$. If one introduces an abstract operator $\mathcal{F}$ (defined by the elastic scattering system described earlier) which sends the scatterer $D$ to the corresponding far-field pattern $u_\tau^\infty$, then the (IP) can be formulated as the following operator equation
\begin{equation}\label{eq:oe}
\mathcal{F}(D)=u_\tau^\infty(\hat x; d, d^\perp,\alpha, \beta, \omega).
\end{equation}
Due to the multiple scattering interaction if more than one scatterer is presented, it is easily seen that \eqref{eq:oe} is {\it nonlinear}, and moreover it is widely known to be {\it ill-posed} in the Hadamard sense. For the measurement data $u_\tau^\infty(\hat x; d, d^\perp,\alpha, \beta, \omega)$ in \eqref{eq:oe}, we always assume that they are collected for all $\hat{x}\in\mathbb{S}^2$. On the other hand, it is remarked that $u_\tau$ is a real-analytic function on $\mathbb{S}^2$, and hence if it is known on any open portion of $\mathbb{S}^2$, then it is known on the whole sphere by analytic continuation. Moreover, if the data set is given for a single quintuplet of $(d, d^\perp, \alpha,\beta, \omega)$, then it is called {\it a single far-field pattern}, otherwise it is called {\it multiple far-field patterns}. Physically, a single far-field pattern can be obtained by sending a single incident plane wave and then measuring the scattered wave field far away in every possible observation direction.

There is a vast literature on the inverse elastic scattering problem as described above. We refer to the theoretical uniqueness results proved in \cite{HH, MD,NU1,NU2,NU3,NKU,R} and,
the sampling-type reconstruction methods for impenetrable elastic bodies developed in \cite{AK,Arens} and those for penetrable ones in \cite{Fact, Sevroglou}. Note that in the above works, both $u_p^\infty$ and $u_s^\infty$ are needed for infinitely many incident plane waves, namely infinitely many far-field patterns are needed. In recent studies in \cite{AA,Amm2,Amm3,Amm4,Amm1,Amm5,Amm6} for reconstructing small elastic inclusions and in \cite{GS,HKJ} for reconstructing extended rigid scatterers, one may implement one type of elastic waves, but still with multiple or even infinitely many far-field patterns. Based on the reflection principle for the Navier system under the third or fourth kind boundary conditions, a global uniqueness with a single far-field pattern was shown in \cite{EY09} for bounded impenetrable elastic bodies of polyhedral type. However, the uniqueness proof there does not apply to the more practical case of rigid bodies and cavities.  Some further discussions on the uniqueness with one or several incident plane waves are mentioned in Section \ref{Scheme-R} and Lemma \ref{Lem:uniqueness}. Using a single set of boundary data,  an extraction formula of an unknown linear crack or the convex hull of an unknown polygonal cavity in $\mathbb{R}^2$ was established in \cite{II07, II09} by means of the enclosure method introduced by Ikehata \cite{Ikehata}.

In this work, we shall consider the inverse problem \eqref{eq:oe} with a single measurement of the P-part far-field pattern $u_p^\infty$, or the S-part far-field pattern $u_s^\infty$, or the full far-field pattern $u^\infty$. According to our earlier discussion, this is an extremely challenging problem with very little theoretical and computational progress in the literature. Moreover, we shall consider our study in a very general but practical setting. There might be multiple target scatterers presented, and the number of the scatters is not required to be known in advance. Furthermore, the scatterers might be of multiple size-scales (in terms of the detecting wavelength); that is, there might be both small-size and regular-size (extended) scatterers presented simultaneously. We develop inverse scattering schemes to locate all the scatterers in a very effective and efficient manner. Specifically, there are three schemes, named as Schemes S, R and M, respectively, proposed for locating multiple small, extended and multiscale rigid scatterers. The core of these schemes is a series of indicator functions, which are directly computed with a single set of far-field data. For Scheme S of locating small scatterers, the inverse problem \eqref{eq:oe} can be linearized by taking the leading term of the relevant far-field expansion with respect to the small size-parameter. For Scheme R of locating extended scatterers, we need to impose a certain a priori knowledge by requiring the shapes of the underlying scatterers are from a certain admissible class that is known in advance. The indicator functions for Scheme R are given by projecting the measured far-field pattern onto the space of the far-field patterns from the admissible scatterers. Finally, a {\it local tuning technique} is implemented to concatenate Schemes S and R to yield Scheme M of locating multiscale scatterers. We would like to remark in passing that our current study follows a similar spirit of the locating methods that were recently
proposed in \cite{LLSS, LLW, LLZ} for inverse electromagnetic and acoustic scattering problems in the frequency domain. However, due to the more complicated behaviors of the elastic wave scattering, particularly the coupling of the compressional and shear waves, the current study is carried out in a more subtle and technical manners. Particularly, we design completely different imaging functionals from those developed \cite{LLSS,LLW,LLZ} for electromagnetic and acoustic problems.

The rest of the paper is organized as follows. In Section \ref{section-small}, we first describe Scheme S of locating multiple small scatterers, and then present the theoretical justification. In Section \ref{section-extended}, we first describe Scheme R of locating multiple extended scatterers and then present the corresponding theoretical justification. Section \ref{section-multi-scale} is devoted to Scheme M of locating multiple multiscale scatterers. In Section \ref{Numerics}, numerical experiments are given to demonstrate the effectiveness and the promising features of the proposed inverse scattering schemes. We conclude our study in Section~6 with several remarks.

\section{Locating multiple small scatterers}\label{section-small}
Throughout the rest of the paper, we assume the angular frequency of incidence is $\omega=1$. Then, the wavelength of the pressure wave is $2\pi/k_p=2\pi\sqrt{\lambda+2\mu}/\omega=\mathcal{O}(1)$, whereas the wavelength of the shear wave is $2\pi/k_s=2\pi\sqrt{\mu}/\omega=\mathcal{O}(1)$. Hence, the size of a
scatterer can be expressed in terms of its Euclidean diameter.
In the sequel we write $u_\tau^\infty(\hat{x}; D, d, d^\bot , \omega)$ ($\tau=p, s\;{\rm or}\; \emptyset$) to signify the dependence of  far-field pattern on the rigid scatterer $D$, incident direction $d$, polarization direction $d^\bot$ and incidence frequency $\omega$.
In certain situations we only indicate the dependence of the far-field pattern on $D$ or $\omega$, but the notation shall be clear from the context. Unless otherwise stated, the space $L^2$ always signifies $L^2(\s^2)^3$.

Next, we first describe Scheme S of locating multiple small rigid elastic scatterers and then present the corresponding theoretical justifications.

\subsection{Description of Scheme S}\label{scheme-small}

We first introduce the class of small elastic rigid scatterers. For $l_s\in \N$, let $M_j$, $1\leq j\leq l_s$, be bounded Lipschitz simply-connected domains in $\R^3$. It is supposed that all $M_j$'s contain the origin and their diameters are comparable with the S-wavelength or P-wavelength, i.e.,
${\rm diam}(M_j)\sim \mathcal{O}(1)$ for all $j=1,2,\cdots,l_s$. For $\rho\in \R_+$, we introduce a scaling/dilation operator $\Lambda_\rho $ by
\be\label{scaling}
\Lambda_\rho M_j:=\{\rho x: x\in M_j\}
\en
and set
\be\label{eq:smalls}
 D _j:=z_j+\Lambda_\rho\,M_j,\quad z_j\in\R^3,\quad 1\leq j\leq l_s.
\en
Each $ D_j$ is referred to as a scatterer component located at $z_j$ with the shape $M_j$. The number $\rho$ represents the {\it relative size/scale} of each component.
In the sequel, we shall reserve the letter $l_s$ to denote the number of components of a small scatterer given by
\be\label{Scatterer}
 D =\bigcup_{j=1}^{l_s}  D _j.
\en
For technical purpose, we next make the following qualitative assumption that $\rho\ll 1$ and
\be\label{assumption-small}
L_s=\min_{j\neq j', 1\leq j,j'\leq l_s} {\rm dist}(z_j, z_{j'})\gg 1.
\en
The above assumption means that the size of each scatterer component is small compared to detecting wavelength, and if there are multiple components, they are sparsely distributed. In our numerical experiments in Section~\ref{Numerics}, we could speak a bit more about the qualitative assumption~\eqref{assumption-small}. Indeed, it is shown that as long as the size of the target scatterer is smaller than half a wavelength, and if there are multiple components presented, the distance between different components is bigger than half a wavelength, the proposed locating methods in this article work in an effective manner.

In order to present Scheme S of locating the multiple components of $D$ in \eqref{Scatterer}, we introduce the following three indicator functions $I_{m}(z)$, $z\in\mathbb{R}^3$, $m=1,2,3$, relying on the availability of different types of far-field patterns. Define
\ben
I_{1}(z)&=&\frac{1}{\|u_p^\infty(\hat{x};  D )\|_{L^2}^2}\sum_{j=1}^3 \left|  \left\langle  u_p^\infty(\hat{x};  D ), (\hat{x} \otimes\hat{x})\textbf{e}_j\,e^{-ik_p\hat{x}\cdot z}  \right \rangle  \right|^2,\\
I_{2}(z)&=&\frac{1}{\|u_s^\infty(\hat{x};  D )\|_{L^2}^2}\sum_{j=1}^3 \left|  \left\langle  u_s^\infty(\hat{x};  D ), (\textbf{I}-\hat{x} \otimes\hat{x})\textbf{e}_j\,e^{-ik_s\hat{x}\cdot z}  \right \rangle  \right|^2,\\
I_{3}(z)&=&\frac{1}{\|u^\infty(\hat{x};  D )\|_{L^2}^2}\sum_{j=1}^3 \left|  \left\langle  u^\infty(\hat{x};  D ), (\hat{x} \otimes\hat{x})\textbf{e}_j\,e^{-i k_p\hat{x}\cdot z}+ (\textbf{I}-\hat{x} \otimes\hat{x})\textbf{e}_j\,e^{-i k_s\hat{x}\cdot z}  \right \rangle  \right|^2,
\enn
where and in the following, the notation $\langle\cdot,\cdot\rangle$ denotes the inner product in $L^2=L^2(\mathbb{S}^2)^3$ with respect to the variable $\hat{x}\in\mathbb{S}^2$; the symbol $\hat{x} \otimes\hat{x}:=\hat{x}^\top \,\hat{x}\in \R^{3\times 3}$ stands for the tensor product; $\textbf{I}$ denotes the $3\times 3$ identity matrix; and
\ben
\textbf{e}_1=(1,0,0)^\top,\quad \textbf{e}_2=(0,1,0)^\top,\quad \textbf{e}_3=(0,0,1)^\top,
\enn
are the three Euclidean base vectors in $\R^3$. Obviously, $I_{m}$  $(m=1,2,3)$ are all nonnegative functions and they can be obtained, respectively, by using a single $P$-part far-field pattern ($m=1$), $S$-part far-field pattern ($m=2$), or the full far-field pattern ($m=3$). The functions introduced above possess certain indicating behavior, which lies in the core of Scheme S. Before stating the theorem of the indicating behavior for those imaging functions, we introduce the real numbers
\be\label{K}
K^j_{1}:=\frac{\|u_p^\infty(\hat{x};  D _j)\|_{L^2}^2}{\|u_p^\infty(\hat{x};  D )\|_{L^2}^2},\quad
K^j_{2}:=\frac{\|u_s^\infty(\hat{x};  D _j)\|_{L^2}^2}{\|u_s^\infty(\hat{x};  D )\|_{L^2}^2},\quad
K^j_3:=\frac{\|u^\infty(\hat{x};  D _j)\|_{L^2}^2}{\|u^\infty(\hat{x};  D )\|_{L^2}^2},
\en for $1\leq j\leq l_s$. 
\begin{thm}\label{S-P}
For a rigid elastic scatterer $D$ described in \eqref{scaling}--\eqref{assumption-small}, and $K_m^j$, m=1,2,3, defined in \eqref{K}, we have
\be\label{eq7}
K^j_m=\widetilde{K}^j+\mathcal{O}(L_s^{-1}+\rho),\quad 1\leq j\leq l_s,\;m=1,2,3,
\en
where $\widetilde{K}^j$'s are positive numbers independent of $L_s$, $\rho$ and $m$. Moreover, there exists an open neighborhood of $z_j$, {\rm neigh}($z_j$), such that
\be\label{eq5}
I_{m}(z)\leq \widetilde{K}^j+\mathcal{O}(L_s^{-1}+\rho)\quad\mbox{for all}\quad z\in {\rm neigh}(z_j),
\en and $I_{m}(z)$ achieves its maximum at $z_j$ in {\rm neigh}($z_j$), i.e.,
\be\label{eq6}
I_{m}(z_j)=\widetilde{K}^j+\mathcal{O}(L_s^{-1}+\rho).
\en
\end{thm}
\begin{rem} The local maximizing behavior of $I_m(z)$ clearly can be used to locate the positions of the scatterer components $D$, namely $z_j$, $1\leq j\leq l_s$.  Such indicating behavior is more
evident if one considers the case that $D$
has only one component, i.e., $l_s = 1$. In the
one-component case, one has that
\ben
\widetilde{K}^j=1,\quad    I_{m}(z)<1+\mathcal{O}(\rho)\quad\mbox{for all}\quad m=1,2,3,\; z\neq z_1,
\enn
but
\ben
I_{m}(z_1)=1+\mathcal{O}(\rho),\quad m=1,2,3.
\enn
That is, $z_1$ is a global maximizer for $I_{m}(z)$.
\end{rem}

Based on Theorem \ref{S-P} we can formulate Scheme S to locate the multiple
small scatterer components of $D$ in \eqref{Scatterer} as follows.
\bigskip

\begin{tabularx}{\textwidth }{>{\bfseries}lX}
\toprule
Scheme S&  Locating small scatterers of $D$ in \eqref{Scatterer}.
\\ \midrule
Step 1 & For an unknown rigid scatterer $D$ in (\ref{Scatterer}),
  collect the P-part ($m=1$), S-part ($m=2$) or the full far-field data ($m=3$) by sending
a single detecting plane wave (\ref{plane}).  \\ \midrule
Step 2 &  Select a sampling region with a mesh $\mathcal{T}_h$ containing $D$.\\ \midrule
Step 3 & For each sampling point $z\in\mathcal{T}_h$, calculate $I_{m}(z)$ ($m=1,2,3$) according to the measurement data.\\ \midrule
Step 4 & Locate all the local maximizers of $I_{m}(z)$ on $\mathcal{T}_h$, which represent the locations of the
scatterer components.\\
\bottomrule
\end{tabularx}

\medskip

\begin{rem}\label{rem:22} In practice, the compressional wave number $k_p= \omega/\sqrt{\lambda+2\mu}$ is smaller than the shear wave number $k_s= \omega /\sqrt{\mu}$. 
 Hence,  the P-wavelength $2\pi/k_p$ is usually larger than the S-wavelength $2\pi/k_s$.
This suggests that using the shear wave measurement would yield better reconstruction than using the compressional wave measurement for locating the multiple small scatterers.  That is, the indicator function $I_{2}$ would work better than $I_{1}$ for the reconstruction purpose, especially when the L\'{a}me constant $\lambda$ is very large compared to $\mu$.  This also suggests that the reconstruction using the indicator function $I_{3}$ with the full far-field pattern will be more stable (w.r.t.  noise) and reliable than the other two; see also Section \ref{Numerics}.
\end{rem}


\subsection{Proof of Theorem \ref{S-P}}\label{section-S-P}
In this section, we shall provide the proof for Theorem~\ref{S-P}. First, we recall
the fundamental solution (Green's tensor) to the Navier equation (\ref{Navier}) given by
\be\label{Pi}
\Pi(x,y)=\Pi^{(\omega)}(x,y)=\frac{k_s^2}{4\pi\omega^2}\frac{e^{ik_s|x-y|}}{|x-y|} \textbf{I}+\frac{1}{4\pi \omega^2}\, \grad_x\,\grad_x^\top\;\left[\frac{e^{ik_s|x-y|}}{|x-y|}-\frac{e^{ik_p|x-y|}}{|x-y|} \right]\!\!,
\en for $x, y\in \R^3$, $x\neq y$.
In order to prove Theorem \ref{S-P} we shall need the following critical lemma on the asymptotic behavior of the elastic far-field patterns due to small scatterers.

\begin{lem}\label{lem2} Let the incident plane wave be given in (\ref{plane}) and $D$ be given in \eqref{scaling}--\eqref{assumption-small}.
 The P-part and S-part far-field patterns have the following asymptotic expressions as $\rho/L_s\rightarrow +0$:
\ben
u_p^\infty(\hat{x}; D)&=&\frac{\rho}{4\pi (\lambda+2\mu)}(\hat{x}\otimes\hat{x})\left[ \sum_{j=1}^{l_s}  e^{-ik_p \hat{x}\cdot z_j}\,(C_{p,j}\,\alpha \,e^{ik_pz_j\cdot d}+C_{s,j}\,\beta\,e^{ik_sz_j\cdot d})  \right]\\
&&+\mathcal{O}\left(\rho^2\,l_s(1+L_s^{-1}) \right),\\
u_s^\infty(\hat{x}; D)&=&\frac{\rho}{4\pi \mu}(\textbf{\emph{I}}-\hat{x}\otimes\hat{x})\left[ \sum_{j=1}^{l_s}  e^{-ik_s \hat{x}\cdot z_j}\,(C_{p,j}\,\alpha \,e^{ik_pz_j\cdot d}+C_{s,j}\beta\,e^{ik_sz_j\cdot d})  \right]\\
&&+\mathcal{O}\left(\rho^2\,l_s(1+L_s^{-1}) \right),
\enn
where $C_{p,j}, C_{s,j}\in \C^{3}$ are constant vectors independent of $\rho, l_s, L_s$ and $z_j$.
\end{lem}
The proof of Lemma \ref{lem2} relies essentially on the asymptotic expansions of $u_p^\infty$ and $u_s^\infty$ in the recent work \cite{Sini}, where
 the Lax-Foldy formulations for the Lam\'{e} system was justified without the condition (\ref{assumption-small}).
 The other references on the asymptotic expansions associated with small inclusions can be found in a series of works by H.
Ammari and H. Kang and their collaborators using integral equation methods; see e.g. \cite{Amm1, Amm2, Amm3, Amm4, Amm5, Amm6}. We also mention the monographs \cite{Martin} by P. Martin where the multiple scattering issues are well treated and \cite{DK} for analysis of acoustic, electromagnetic and elastic scattering problems at low frequencies. For the reader's convenience, we present a proof of Lemma~\ref{lem2} under the sparsity assumption (\ref{assumption-small}).

\begin{proof}[Proof of Lemma~\ref{lem2}]

By \cite[Remark 1.3]{Sini}, there exists a small number $\epsilon>0$ such that for $(l_s-1)\rho/L_s<\epsilon$
\be\label{eq3}\begin{split}
u_p^\infty(\hat{x}; D)=&\frac{1}{4\pi (\lambda+2\mu)}(\hat{x}\otimes\hat{x})\left[ \sum_{j=1}^{l_s} c_j e^{-ik_p \hat{x}\cdot z_j}\,Q_j  \right]+\mathcal{O}\left(\rho^2\,l_s(1+L_s^{-1}) \right),\\
u_s^\infty(\hat{x}; D)=&\frac{1}{4\pi \mu}(\textbf{I}-\hat{x}\otimes\hat{x})\left[ \sum_{j=1}^{l_s} c_j e^{-ik_s \hat{x}\cdot z_j}\,Q_j  \right]+\mathcal{O}\left(\rho^2\,l_s(1+L_s^{-1}) \right),
\end{split}
\en
where the vector coefficients $Q_j\in \C^{3}$, $j=1,2,\cdots,l_s$ are the unique solutions to the linear algebraic system
\be\label{eq2}
C_j^{-1}\, Q_j=-u^{in}(z_j)-\sum_{m=1,m\neq j}^{l_s} \Pi^{(\omega)}(z_j,z_m)\, Q_m,
\en
with $\Pi^{(\omega)}(z_j,z_m)$ denoting the Kupradze matrix (\ref{Pi}) and
$$
C_j:=\int_{\partial D_j} \Theta_j(y)ds(y)\in \C^{3\times 3}.
$$
Here, $\Theta_j$ is the solution matrix of the first kind integral equation
\be\label{eq1}
\int_{\partial D_j} \Pi^{(0)}(x,y) \Theta_j(y) ds(y)=\textbf{I},\quad x\in \partial D_j,
\en
where the matrix $\Pi^{(0)}(x,y)$, which denotes the the Kelvin matrix of the fundamental solution of the Lam\'e system with $\omega=0$, takes the form (see,  e.g., \cite[Chapter 2]{Kupradze} or \cite[Chapter 2.2]{Hsiao} )
\be\label{G03}
\Pi^{(0)}(x,y):=\frac{\lambda+3\mu}{8\pi \mu (\lambda+2\mu)} \frac{1}{|x-y|}\,\textbf{I}+\frac{\lambda+\mu}{8\pi\mu(\lambda+2\mu)}\frac{1}{|x-y|^3}\,\big((x-y)\otimes (x-y)\big)
\en
Since $\Pi^{(0)}(x,y)\sim |x-y|^{-1}$ as $x\rightarrow y$, it follows from (\ref{eq1}) that $\Theta_j(y)\sim \rho^{-3}$ for $y\in \partial D_j$, from which we get $C_j\sim \rho^{-1}$ for $j=1,2\cdots,l_s$ as $\rho\rightarrow +0$. Now, inserting the estimate of $C_j$ into (\ref{eq2}) and taking into account the fact that
\ben
\Pi^{(\omega)}(z_j,z_m)=\mathcal{O}(L_s^{-1})\quad \mbox{for}\quad j\neq m,
\enn
we obtain
\ben
Q_j=\rho\,\textbf{H}_j\, u^{in}(z_j)+\mathcal{O}(L_s^{-1}+\rho^2)\quad \mbox{as}\quad \rho/L_s\rightarrow +0,\quad j=1,2\cdots l_s,
\enn
where $\textbf{H}_j\in \C^{3\times 3}$ are some constant matrices independent of $\rho$ and $L$. Therefore, Lemma \ref{lem2} is proved by taking
$C_{p,j}=\textbf{H}_j\, d$,  $C_{s,j}=\textbf{H}_j\, d^\bot$.
\end{proof}

We are in a position to present the proof of Theorem \ref{S-P}.

\begin{proof}[Proof of Theorem \ref{S-P}]
We first consider the indicating function $I_{1}(z)$, $z\in\mathbb{R}^3$.
For notational convenience we write
\ben
A_j=A_j(z_j, \alpha,\beta):=C_{p,j}\,\alpha \,e^{ik_pz_j\cdot d}+C_{s,j}\,\beta\,e^{ik_sz_j\cdot d}\in \C^{3},\quad j=1,2,\cdots,l_s,
\enn with $C_{p,j}$, $C_{s,j}$ given as in Lemma \ref{lem2}.
Then, it is seen from Lemma \ref{lem2} that
\ben
\| u_p^\infty(\hat{x}; D)\|^2_{L^2}&=&\frac{\rho^2 }{4 (\lambda+2\mu)^2}\sum_{j=1}^{l_s} |A_j|^2+\mathcal{O}\left(\rho^3+\rho^2L^{-1}\right),\\
\| u_p^\infty(\hat{x}; D_j)\|^2_{L^2}&=&\frac{\rho^2 }{4 (\lambda+2\mu)^2}|A_j|^2+\mathcal{O}\left(\rho^3+\rho^2L_s^{-1}\right).
\enn Hence,
\be\label{Kj}
K_1^j=\frac{\| u_p^\infty(\hat{x}; D_j)\|^2_{L^2}}{\| u_p^\infty(\hat{x}; D)\|^2_{L^2}}=\widetilde{K}^j+\mathcal{O}(\rho+L_s^{-1}),\quad
\widetilde{K}^j:=\frac{A_j^2}{\sum_{j=1}^{l_s} |A_j|^2}.
\en
This proves (\ref{eq7}) for $m=1$. The case of using the S-part of the far-field pattern (i.e., $m=2$) can be treated  in an analogous way.

For the full-wave scenaio, namely when $m=3$, the orthogonality of $u_p^\infty$ and $ u_s^\infty$ should be used in the treatment. Since
$\langle \textbf{I}-\hat{x}\otimes\hat{x}, \hat{x}\otimes\hat{x} \rangle=0$,
 by applying Lemma \ref{lem2} again to $D$ and $D_j$, we have
\ben
\| u^\infty(\hat{x}; D)\|^2_{L^2}
&=&\frac{\rho^2 }{4}\left(\frac{1}{(\lambda+2\mu)^2}+\frac{1}{\mu^2} \right)  \sum_{j=1}^{l_s}  |A_j|^2+\mathcal{O}\left(\rho^3+\rho^2L_s^{-1}\right),\\
\| u^\infty(\hat{x}; D_j)\|^2_{L^2}
&=&\frac{\rho^2 }{4}\left(\frac{1}{(\lambda+2\mu)^2}+\frac{1}{\mu^2} \right)  |A_j|^2+\mathcal{O}\left(\rho^3+\rho^2L_s^{-1}\right).
\enn
Hence, the equality (\ref{eq7}) with $m=3$ is proved with the same $\widetilde{K}^j$ given in (\ref{Kj}).

To verify (\ref{eq5}) and (\ref{eq6}),
without loss of generality we only consider the indicating behavior of $I_{1}(z)$ in a small neighborhood of $z_j$ for some fixed $1\leq j\leq l_s$, i.e., $z\in{\rm neigh}(z_j)$. Clearly, under the assumption (\ref{assumption-small}) we have
\ben
\omega|z_{j'}-z|\sim\omega\,L_s\gg 1,\quad\mbox{for all}~z\in {\rm neigh}(z_j),\quad j'\neq j.
\enn
By using the Reimann-Lebesgue lemma about oscillating integrals and Lemma \ref{lem2} we can obtain
\be\no
&&\left|\left\langle u_p^\infty(\hat{x}; D), \sum_{j=1}^3(\hat{x}\otimes\hat{x})\textbf{e}_j e^{-ik_p\hat{x}\cdot z}\right\rangle\right|^2 \\ \no
&=&\frac{\rho^2\, |A_j|^2}{16\pi^2 (\lambda+2\mu)^2} \;\left\langle e^{-ik_p\hat{x}\cdot z_j}, e^{-ik_p\hat{x}\cdot z}  \right\rangle
+\mathcal{O}(\rho^3\,+\rho^2L_s^{-1})\\ \label{eq4}
&\leq & \frac{\rho^2\,|A_j|^2}{4 (\lambda+2\mu)^2} +\mathcal{O}(\rho^3\,+\rho^2L_s^{-1}),
\en
where the last inequality follows from the Cauchy-Schwartz inequality. Moreover, the strict inequality in (\ref{eq4}) holds if $z\neq z_j$ and the equal sign holds only when $z=z_j$. Therefore, by the definition of $I_{1}$,
\ben
I_1(z)\leq \widetilde{K}^j+\mathcal{O}(\rho+L_s^{-1}),
\enn
and only when $z=z_j$ the equality holds. This proves (\ref{eq5}) and (\ref{eq6}). The indicating behavior of $I_2$ and $I_3$ can be treated in the same manner.

The proof is completed.
\end{proof}

\section{Locating multiple extended scatterers}\label{section-extended}
In this section we consider the locating of multiple rigid scatterers of regular size by using a single incident plane wave. As discussed earlier in Introduction, it is extremely challenging to recover a generic rigid elastic scatterer by using a single far-field pattern. The scheme that we shall propose for locating multiple extended (namely, regular-size) scatterers requires a certain a priori knowledge of the underlying target objects; that is, their shapes must be from a certain known class. In what follows, we first describe the multiple extended scatterers for our study and then present the corresponding locating Scheme R.

For $j=1,2,\cdots,l_e$, set $r_j\in \R_+$ such that
\ben
r_j\in[R_0, R_1], \quad 0<R_0<R_1<+\infty,\quad R_0 \sim\mathcal{O}(1).
 \enn
 Let $E_j\subset \R^3$, $1\leq j\leq l_e$ denote a bounded simply-connected Lipschitz domain containing the origin. Throughout, we assume that ${\rm diam} (E_j)\sim 1$, $1\leq j\leq l_e$. Define the scaling operator $\Lambda_{r} E_j$ with $r\in \R_+$ to be the same one as that given in \eqref{scaling}.
Denote by $\mathcal{R}_{j}:=\mathcal{R}(\theta_j,\phi_j,\psi_j)\in SO(3)$, $1\leq j\leq l_e$, the 3D rotation matrix around the origin whose Euler angels are $\theta_j\in[0,2\pi],\phi_j\in[0,2\pi]$ and $\psi_j\in[0,\pi]$; and define $\mathcal{R}_j E:=\{\mathcal{R}_j x: x\in E\}$. For $z_j\in \R^3$, we let
\be\label{Omega}
\Omega=\bigcup_{j=1}^{l_e} \Omega_j,\quad \Omega_j:=z_j+\mathcal{R}_{j}\, \Lambda_{r_j}\, E_j,
\en
denote the extended target scatterer for our current study.
Obviously, $\Omega$ is a collection of scatterer components $\Omega_j$ that obtained by scaling, rotating and translating $E_j$ with the parameters $r_j, (\theta_j,\phi_j,\psi_j)$ and $z_j$, respectively. In the sequel, the parameter $z_j$, Euler angles $(\theta_j,\phi_j,\psi_j)$, number $r_j$ and the reference scatterer $E_j$ will be respectively referred to as the \emph{position}, \emph{orientation}, \emph{size} and \emph{shape} of the scatterer component $\Om_j$ in $\Omega$. For technical purpose, we impose the following sparsity assumption on the extended scatterer $\Omega$ introduced in \eqref{Omega},
\begin{equation}\label{eq:rs}
L_e=\min_{j\neq j', 1\leq j,j'\leq l_e} {\rm dist}(\Omega_j, \Omega_{j'})\gg 1.
\end{equation}
Furthermore, it is assumed that there exists an admissible reference scatterer space
\be\label{E0}
\mathscr{A}:=\{\Sigma_j\}_{j=1}^{l'}
\en
where each $\Sigma_j\subset \R^3$ is a bounded simply-connected Lipschitz domain containing the origin, such that for $\Omega$ in \eqref{Omega},
\begin{equation}\label{eq:adms}
E_j\in\mathscr{A}.
\end{equation}
For the admissible reference space $\mathscr{A}$ introduced in \eqref{E0}, we require that
\be\label{eq:admss}
\Sigma_j\neq \Sigma_{j'}\quad\mbox{for}\quad j\neq j',\quad 1\leq j,j'\leq l',
\en
and it is known in advance. The number $l'\in\mathbb{N}$ in \eqref{E0} is not necessarily equal to $l_e$ in \eqref{Omega}. Condition \eqref{eq:adms} implies that the shapes of target scatterer components must be known in advance. Nevertheless, it may happen that more than one scatterer component possesses the same shape, or some shapes from the admissible class $\mathscr{A}$ may not appear in the target scatterer components.

In the following, we shall develop Scheme R by using a single far-field pattern to locate the multiple components of the scatterer $\Omega$ described above.
%
%
The inverse problem could find important practical applications in the real world. For instance, in locating an unknown group of plastic-cased land mines or pipelines buried in dry soils, one has the a priori knowledge on the possible shapes of the target objects.


\subsection{Description of Scheme R}\label{Scheme-R}

For $h\in\mathbb{R}_+$, $h\ll 1$, let $\mathscr{N}_1$ be a suitably chosen finite index set such that $\{\mathcal{R}_j\}_{j\in\mathscr{N}_1}=\{\mathcal{R}(\theta_j,\phi_j,\psi_j)\}_{j\in\mathscr{N}_1}$ is an $h$-net of $SO(3)$. That is, for any rotation matrix $\mathcal{R}\in SO(3)$, there exists $j\in\mathscr{N}_1$ such that $\|\mathcal{R}_j-\mathcal{R}\|\leq h$. For a simply-connected domain $\Sigma$ containing the origin, we define
\begin{equation}\label{eq:a1}
\mathcal{R}_h \Sigma:=\{\mathcal{R}_j \Sigma\}_{j\in\mathscr{N}_1}.
\end{equation}
In an analogous manner, for $\Lambda_r$ with $r\in [R_0, R_1]$, we let $\mathscr{N}_2$ be a suitably chosen finite index set such that $\{r_j\}_{j\in\mathscr{N}_2}$ is an $h$-net of $[R_0, R_1]$. Define
\begin{equation}\label{eq:a1}
\Lambda_{h} \Sigma:=\{\Lambda_{r_j} \Sigma\}_{j\in\mathscr{N}_2}.
\end{equation}
Next, we augment the admissible reference space $\mathscr{A}$ to be
\begin{equation}\label{eq:a3}
\mathscr{A}_h=\mathcal{R}_h\Lambda_h \mathscr{A}=\bigcup_{j=1}^{l'}\{\mathcal{R}_h\Lambda_h \Sigma_j\}:=\{\widetilde{\Sigma}_j\}_{j=1}^{l''},
\end{equation}
where $l''$ denotes the cardinality of the discrete set $\mathscr{A}_h$. Indeed, $\mathscr{A}_h$ can be taken as an $h$-net of $\mathscr{A}$ in the sense that for any $\Sigma\in\mathscr{A}$, there exists $\widetilde{\Sigma}\in\mathscr{A}_h$ such that $d_H(\overline\Sigma,\overline{\widetilde\Sigma})\leq C h$, where $d_H$ denotes the Hausdorff distance and $C$ is a positive constant depending only on $\mathscr{A}$. We shall make the following two assumptions about the augmented admissible reference space $\mathscr{A}_h$:
\begin{description}
\item[(i)] $u_\tau^\infty(\hat{x}; \widetilde{\Sigma}_j)\neq u_\tau^\infty(\hat{x}; \widetilde{\Sigma}_{j'})$ for $\tau=s, p$ or $\emptyset$, and $j\neq j'$, $1\leq j,j'\leq l''$.
\item[(ii)] $\|u_\tau^\infty(\hat{x}; \widetilde{\Sigma}_j)\|_{L^2}\geq \|u_\tau^\infty(\hat{x}; \widetilde{\Sigma}_{j'})\|_{L^2}$ for $\tau=s, p$ or $\emptyset$, and $j<j'$, $1\leq j,j'\leq l''$.
\end{description}
Assumption (ii) can be fulfilled by reordering the elements in $\mathscr{A}_h$ if necessary.
For assumption (i), we recall the following well-known conjecture in the inverse elastic scattering theory:
\be\label{uniqueness}
u_\tau^\infty(\hat{x}; D_1)=u_\tau^\infty(\hat{x}; D_2)\quad\mbox{for all}\quad \hat{x}\in \mathbb{S}^2\quad\mbox{if and only if}\quad D_1=D_2,
\en
where $D_1$ and $D_2$ are two rigid elastic scatterers. \eqref{uniqueness} states that one can uniquely determine an elastic rigid scatterer by using a single far-field pattern. There is a widespread belief that \eqref{uniqueness} holds true, but there is  very limited progress in the literature, and still largely remains open. We refer to \cite{HH,MD,GS} for uniqueness results established by using infinitely many far-field measurements, and \cite{HKJ} for uniqueness in determining spherical or convex polyhedral rigid scatterers by using a single S-part far-field pattern. Nevertheless, since $\mathscr{A}_h$ is known, assumption (i) can be verified in advance.


In order to identify the multiple extended scatterers of $\Omega$ in (\ref{Omega}), we introduce the following $l''\times 3$ indicator functions:
\be\no
W^j_1(z)&=&\frac{1}{\|u_p^\infty(\hat{x}; \widetilde{\Sigma}_j)\|^2_{L^2}}\,\left|\left\langle  u_p^\infty(\hat{x}; \Om),\; e^{-ik_p\hat{x}\cdot z}\,u_p^\infty(\hat{x}; \widetilde{\Sigma}_j)        \right\rangle \right|^2,\\ \label{W}
W^j_2(z)&=&\frac{1}{\|u_s^\infty(\hat{x}; \widetilde{\Sigma}_j)\|^2_{L^2}}\,\left|\left\langle  u_s^\infty(\hat{x}; \Om),\; e^{-ik_s\hat{x}\cdot z}\,u_s^\infty(\hat{x}; \widetilde{\Sigma}_j)        \right\rangle \right|^2,\\ \no
W^j_3(z)&=&\frac{1}{\|u^\infty(\hat{x}; \widetilde{\Sigma}_j)\|^2_{L^2}}\,\left|\left\langle  u^\infty(\hat{x}; \Om),\; e^{-ik_p\hat{x}\cdot z}\,u_p^\infty(\hat{x}; \widetilde{\Sigma}_j)+e^{-ik_s\hat{x}\cdot z}\,u_s^\infty(\hat{x}; \widetilde{\Sigma}_j)        \right\rangle \right|^2,
\en  where $z\in\mathbb{R}^3$ and $\widetilde{\Sigma}_j\in {\mathscr{A}}_h$ for $j=1,2,\cdots l''$.

Next, we present a key theorem on the indicating behavior of these indicator functions, which
forms the basis of our Scheme R. Recall that $\alpha,\beta$ are the coefficients attached to $u^{in}_p$ and $u^{in}_s$, respectively, in the expression of $u^{in}$ given in (\ref{plane}).
\begin{thm}\label{Th2}
Suppose that $\alpha\,\beta=0$ and that $\widetilde{\Sigma}_1\in \mathscr{A}_h$ is of the following form
\begin{equation}\label{eq:b1}
\widetilde{\Sigma}_1=\mathcal{R}_{j_{\sigma}}\Lambda_{r_{j_{\tau}}}\Sigma_{j_0},\quad \Sigma_{j_0}\in\mathscr{A},\ j_{\sigma}\in\mathscr{N}_1,\ j_{\tau}\in\mathscr{N}_2.
\end{equation}
Suppose that in $\Omega$ given by \eqref{Omega}, there exists $J_0\subset\{1,2,\ldots,l_e\}$ such that for $j\in J_0$, the component $\Omega_j=\mathcal{R}_j\Lambda_{r_j} E_j$ satisfies
\begin{equation}\label{eq:b2}
(i)~E_j=\Sigma_{j_0};\ \ (ii)~\|\mathcal{R}_j-\mathcal{R}_{j_{\sigma}}\|\leq h;\ \ (iii)~\|r_{j}-r_{j_{\tau}}\|\leq h;
\end{equation}
whereas for $j\in\{1,2,\ldots,l_e\}\backslash J_0$, at least one of the conditions in \eqref{eq:b2} is not fulfilled by the scatterer component $\Omega_j$. Then for each $z_j$, $1\leq j\leq l_e$, there exists an open neighborhood of $z_j$, $neigh(z_j)$, such that
\begin{enumerate}[(i)]
\item if $j\in J_0$, then
\begin{equation}\label{eq:b3}
W_m^1(z)\leq 1+\mathcal{O}\left(\frac 1 L_e+h\right),\quad \forall z\in {\rm neigh}(z_j),\ \ m=1,2,3.
\end{equation}
Moreover, the equality relation holds in \eqref{eq:b3} only when $z=z_j$. That is, $z_j$ is a local maximum point for $W_m^1(z)$.

\item if $j\in\{1,2,\ldots,l\}\backslash J_0$, then there exists $\epsilon_0\in\mathbb{R}_+$ such that
\begin{equation}\label{eq:b4}
W_m^1(z)\leq 1-\epsilon_0+\mathcal{O}\left(\frac{1}{L_e}+h\right), \quad \forall z\in {\rm neigh}(z_j),\ \ m=1,2,3.
\end{equation}
\end{enumerate}

%
\end{thm}
\begin{rem}
The condition $\alpha\beta=0$ implies that  $W_m^j$'s in (\ref{W}) are valid for incident plane pressure or shear waves only. Following the proof of Theorem \ref{Th2}, one can formulate the indicator functions for general elastic plane waves of the form (\ref{plane}); see Remark \ref{Re:Ex} at the end of this section.
\end{rem}
In Theorem~\ref{Th2}, it may happen that $J_0=\emptyset$. In this case, there is no scatterer in $\Omega$ which is of the shape $\Sigma_{j_0}$. Clearly, by using the indicating behavior of the functional $W_m^1(z)$ presented in Theorem~\ref{Th2}, one can locate all the scatterer components possessing the shape $\Sigma_{j_0}$. After the locating of those scatterers of the shape $\Sigma_{j_0}$, one can exclude them from the searching region. Moreover, by Lemmas~\ref{lem1} and \ref{translation-relation} in the following, one can calculate the far-field pattern generated by the remaining scatterer components. With the updated far-field pattern, one can then use $\widetilde{\Sigma}_2$ as the reference, and proceed as before to locate all the scatterer components of $\Omega$ possessing the same shape as $\widetilde{\Sigma}_2$. Clearly, this procedure can be carried out till we find all the scatterer components of $\Omega$. In summary, Scheme R is read as follows.

\bigskip

\begin{tabularx}{\textwidth }{>{\bfseries}lX}
\toprule
Scheme R&  Locating extended scatterers of $\Omega$ in \eqref{Omega}.
\\ \midrule
Step 1& For the admissible reference scatterer class $\mathscr{A}$ in (\ref{E0}), formulate the augmented admissible class $\mathscr{A}_h$ in (\ref{eq:a3}). \\ \midrule
Step 2 & Collect in advance the P-part ($m=1$), S-part ($m=2$) or the full far-field data ($m=3$) associated with the admissible reference scatterer class ${\mathscr{A}}_h$ corresponding to a single incident plane wave of the form (\ref{plane}). Reorder ${\mathscr{A}}_h$ if necessary to make it satisfy assumption (ii), and also verify the generic assumption (i). \\ \midrule
Step 3 & For an unknown rigid scatterer $\Omega$ in (\ref{Omega}),
  collect the P-part, S-part or the full far-field data by sending the same detecting plane wave as specified in {\bf Step 2}.  \\ \midrule
Step 4 & Select a sampling region with a mesh $\mathcal{T}_h$ containing $\Omega$. \\ \midrule
Step 5 & Set $j=1$. \\ \midrule
Step 6 & For each sampling point $z\in\mathcal{T}_h$, calculate $W^j_{m}(z)$ ($m=1,2,3$) according to available far-field data for $\Omega$.\\ \midrule
Step 7 & Locate all those significant local maximum points of $W^j_{m}(z)$ satisfying $W^j_{m}(z)\approx 1$ for the scatterer
components of the form $z+\widetilde{\Sigma}_j$. Let $z_\eta$, $\eta=1,\ldots,\eta_0$ be the local maximum points found this step.\\ \midrule
Step 8 & Remove all those $z+\widetilde{\Sigma}_j$ found in {\bf Step 6} from the mesh $\mathcal{T}_h$.  \\ 
\bottomrule
\end{tabularx}

\begin{tabularx}{\textwidth }{>{\bfseries}lX}
\toprule
Scheme R&  Continue
\\ \midrule
Step 9 & Update the far-field patterns according to the following formulae
\ben
u_p^{\infty, new}&=&u_p^\infty(\hat x; d,d^\bot, \alpha,\beta, \Omega)- u_p^\infty(\hat x; d,d^\bot, \alpha,0,\widetilde{\Sigma}_j)\sum_{\eta=1}^{\eta_0} e^{ik_p (d-\hat{x})\cdot z_\eta}\\
&&- u_p^\infty(\hat x; d,d^\bot, 0,\beta,\widetilde{\Sigma}_j)\sum_{\eta=1}^{\eta_0} e^{i(k_sd- k_p\hat{x})\cdot z_\eta},\\
u_s^{\infty, new}&=&u_s^\infty(\hat x; d,d^\bot, \alpha,\beta, \Omega)- u_s^\infty(\hat x; d,d^\bot, \alpha,0,\widetilde{\Sigma}_j)\sum_{\eta=1}^{\eta_0} e^{i(k_p d-k_s\hat{x})\cdot z_\eta}\\
&&- u_s^\infty(\hat x; d,d^\bot, 0,\beta,\widetilde{\Sigma}_j)\sum_{\eta=1}^{\eta_0} e^{ik_s(d-\hat{x})\cdot z_\eta},\\
 u^{\infty, new}&=&u_p^{\infty, new}+u_s^{\infty, new}.
\enn \\ \midrule
Step 10 & If $j=l''$, namely, the maximum number of the reference scatterers reaches, then stop the reconstruction; otherwise set $j=j+1$, and go to {\bf Step 6}. \\
\bottomrule
\end{tabularx}

\subsection{Proof of Theorem \ref{Th2}}\label{Proof-Th2}

%

Throughout the present section, we let $\nu$ denote the unit normal vector to $\partial\Omega$ pointing into $\R^3\backslash\overline{\Omega}$.
 Denote the linearized strain tensor  by
\begin{equation} \label{strain_2}
\varepsilon(u)\ :=\ \frac{1}{2}\bigl(\nabla u+\nabla u^\top\bigr)\ \in\ \R^{3\times 3},
\end{equation}
 where $\nabla u\in\R^{3\times 3}$ and
$\nabla u^\top$ stand for the Jacobian matrix of $u$ and its adjoint, respectively.
By Hooke's law the strain tensor is related to the stress tensor via the identity
\begin{equation} \label{stress_2}
\sigma(u)\ =\ \lambda\,(\divv u)\,\textbf{I}\ +\ 2\mu\,\varepsilon(u)\ \in\
\R^{3\times 3}.
\end{equation}
The surface traction (or the stress operator) on $\partial \Omega$ is given by
\be\label{traction}
T_{\nu} u:=\sigma(u)\nu=(2\mu \nu\,\cdot \grad+\la\,\nu\,\divv+\mu\nu\,\times \curl)\,u.
\en


We next present several auxiliary lemmas.

\begin{lem}\label{lem1}
Let $\Omega$ be a scatterer with multiple components given in (\ref{Omega}). Under the assumption (\ref{assumption-small}), we have
\ben
u^\infty(\hat{x}; \Omega)=\sum_{j=1}^{l_e}\,u^\infty(\hat{x}; \Omega_j)+\mathcal{O}(L_e^{-1}).
\enn
\end{lem}
\begin{proof} For simplicity we assume $l_e=2$.
We begin with the single- and double-layer potential operators in elasticity. For $j=1,2$, let
\begin{align}
(S_j \varphi)(x):=&2\int_{\partial \Omega_j} \Pi(x,y)\varphi(y)ds(y),\quad \varphi\in C(\partial \Omega_j),\quad x\in \partial \Omega_j,\;\label{eq:sl}\\
(K_j \varphi)(x):=&2\int_{\partial \Omega_j}  \Xi(x,y)\varphi(y)ds(y),\quad\varphi\in C(\partial \Omega_j),\quad x\in \partial \Omega_j,\label{eq:dl}
\end{align}
where $\Xi(x,y)$ is a matrix-valued function whose $j$-th column vector is defined by
\ben
\Xi(x,y)^\top\, \textbf{e}_j:=T_{\nu(y)}(\Pi(x,y)\,\textbf{e}_j)\quad\mbox{on}\quad\partial\Omega_j,\quad\mbox{for}\; x\neq y,\quad j=1,2,3.
\enn
Recall that the superscript $(\cdot)^\top$ denotes the transpose, $\textbf{e}_j\in \C^{3\times 1}$ the usual cartesian unit vectors and $T_{\nu(y)}$ the stress operator defined in (\ref{traction}).
Under the regularity assumption $\partial \Omega_j\in C^2$, it was proved in \cite{HH} that the scattered field $u^{sc}(x; \Omega_j)$ corresponding to $\Omega_j$ can be represented as
\ben
u^{sc}(x; \Omega_j)=\int_{\partial \Omega_j}  \Xi(x,y)\varphi_j(y)ds(y)+i \,\int_{\partial \Omega_j} \Pi(x,y)\varphi_j(y)ds(y),\quad x\in \R^3\backslash \overline{ \Omega }_j,
\enn
where the density function $\varphi_j\in C(\partial \Omega_j)$ is given by
\ben
\varphi_j=-2(I+K_j+i S_j)^{-1} u^{in}|_{\partial \Omega_j}, \quad j=1,2.
\enn
To prove the lemma for the scatterer $ \Omega =\Omega_1\cup \Omega_2$, we make use of the ansatz
\ben
u^{sc}(x;  \Omega )=\sum_{j=1,2}\left\{\int_{\partial \Omega_j}  \Xi(x,y)\phi_j(y)ds(y)+i \,\int_{\partial \Omega_j} \Pi(x,y)\phi_j(y)ds(y)\right\},\quad x\in \R^3\backslash \overline{ \Omega },
\enn
with $\phi_j\in C(\partial \Omega_j)$. Using the Dirichlet boundary condition $u^{sc}+u^{in}=0$ on each $\partial \Omega_j$,  we obtain the integral equations
\be\label{integral}
\begin{pmatrix}
I+K_1+i S_1 & J_2 \\ J_1 & I+K_2+i S_2
\end{pmatrix}
\begin{pmatrix}
\phi_1 \\ \phi_2
\end{pmatrix}
=-2\begin{pmatrix}
u^{in}|_{\partial \Omega_1} \\ u^{in}|_{\partial \Omega_2}
\end{pmatrix},
\en
where the operators $J_1: C(\partial \Omega_1)\rightarrow C(\partial \Omega_2),
J_2: C(\partial \Omega_2)\rightarrow C(\partial \Omega_1)$ are defined respectively by \ben
J_1 \phi_1&=&2\left\{\int_{\partial \Omega_1}  \Xi(x,y)\phi_1(y)ds(y)+i \,\int_{\partial \Omega_1} \Pi(x,y)\phi_1(y)ds(y)\right\},\quad x\in \partial \Omega_2,\\
J_2 \phi_2&=&2\left\{\int_{\partial \Omega_2}  \Xi(x,y)\phi_2(y)ds(y)+i \,\int_{\partial \Omega_2} \Pi(x,y)\phi_2(y)ds(y)\right\},\quad x\in \partial \Omega_1.
\enn
Since $L_e\gg 1$ (cf. (\ref{eq:rs})), using the fundamental solution (\ref{Pi}), it is  readily to estimate
\ben
\|J_1 \phi_1\|_{C(\partial \Omega_2)}\leq C_1 L_e^{-1}\|\phi_1\|_{C(\partial \Omega_1)},\quad
\|J_2 \phi_2\|_{C(\partial \Omega_1)}\leq C_2 L_e^{-1}\|\phi_2\|_{C(\partial \Omega_2)},\quad C_1, C_2>0.
\enn
Hence, it follows from (\ref{integral}) and the invertibility of $I+K_j+i S_j: C(\partial \Omega_j)\rightarrow C(\partial \Omega_j) $ that
\ben
\begin{pmatrix}
\phi_1 \\ \phi_2
\end{pmatrix}&=&
\begin{pmatrix}
(I+K_1+i S_1)^{-1} & 0 \\ 0 & (I+K_2+i S_2)^{-1}
\end{pmatrix}
\begin{pmatrix}
-2u^{in}|_{\partial \Omega_1} \\ -2u^{in}|_{\partial \Omega_2}
\end{pmatrix}+\mathcal{O}(L_e^{-1})\\
&=&\begin{pmatrix}
\varphi_1 \\ \varphi_2
\end{pmatrix}+\mathcal{O}(L_e^{-1}).
\enn
This implies that
\ben
u^{sc}(x;  \Omega )=u^{sc}(x; \Omega_1)+u^{sc}(x; \Omega_2)+\mathcal{O}(L_e^{-1})\quad\mbox{as}\quad L_e\rightarrow\infty,
\enn
which finally leads   to
\ben
u^\infty(\hat{x};  \Omega )=u^\infty(\hat{x}; \Omega_1)+u^\infty(\hat{x}; \Omega_2)+\mathcal{O}(L_e^{-1}).
\enn
\end{proof}
\begin{rem}
In the proof of Lemma \ref{lem1}, we require that the boundary $\partial\Omega$ is $C^2$ continuous. This is mainly due to the requirements of the mapping properties of the single- and double-layer potential operators (cf. \eqref{eq:sl} and \eqref{eq:dl}) in the proof. This regularity assumption can be relaxed to be Lipschitz continuous by using a similar argument, together with the mapping properties of the layer potential operators defined on Lipschitz surfaces (cf. \cite{McL}).
\end{rem}

In what follows, we shall establish the relation between far-field patterns for translated, rotated and scaled elastic bodies.
For $D\subset\R^3$ and $\textbf{a}=(a_1,a_2,a_3)\in \R^3$,
we write $D_\textbf{a}=\textbf{a}+D$ for simplicity.

\begin{lem}\label{translation-relation} Assume $\partial D$ is Lipschitz.
If $\alpha=1$, $\beta=0$, then
\be\label{Far-field-relation-p}
u^\infty_p(\hat{x}; D_\textbf{a})=u^{\infty}_p(\hat{x}; D)\;e^{ik_p(d-\hat{x})\cdot\textbf{a}},\quad
u^\infty_s(\hat{x}; D_\textbf{a})=u^{\infty}_s(\hat{x}; D)\;e^{i(k_pd-k_s \hat{x})\cdot\textbf{a}}.
\en
If $\alpha=0$, $\beta=1$, then
\be\label{Far-field-relation-s}
u^\infty_p(\hat{x}; D_\textbf{a})=u^{\infty}_p(\hat{x}; D)\;e^{i(k_sd-k_p\hat{x})\cdot\textbf{a}},\quad
u^\infty_s(\hat{x}; D_\textbf{a})=u^{\infty}_s(\hat{x}; D)\;e^{ik_s(d-\hat{x})\cdot\textbf{a}}.
\en
\end{lem}
\begin{proof}  We first consider the case of incident plane pressure waves, i.e., $\alpha=1$, $\beta=0$. Denote by $u^\infty(\hat{x};D)=u^\infty(\hat{x}; D, d)$
the far-field pattern corresponding to the rigid scatter $D$ with the incident direction $d\in \s^2$. For any $y=z+\textbf{a}\in \partial D_\textbf{a}$
with $z\in \partial D$, we have
\ben
u^{sc}(y; D_\textbf{a})=-d \,e^{ik_p z\cdot d}\,e^{ik_p \textbf{a}\cdot d}
=u^{sc}(z; D)\,e^{ik_p \textbf{a}\cdot d}=u^{sc}(y-\textbf{a}; D)\,e^{ik_p \textbf{a}\cdot d}.
\enn
From the uniqueness of the exterior problem of the Navier equation for rigid scatterers (cf.~\cite{HH}), it follows that
\be\label{u1}
u^{sc}(y; D_\textbf{a})=u^{sc}(y-\textbf{a}; D)\,e^{ik_p \textbf{a}\cdot d}=u^{sc}(z; D)\,e^{ik_p \textbf{a}\cdot d},
\quad\mbox{for all}~ y\in \R^3\backslash\overline{D}.
\en 
This implies that for any $y=z+\textbf{a}\in \partial D_\textbf{a}$ with $z\in \partial D$,
\be\label{u2}
T_{\nu(y)} u^{sc}(y; D_\textbf{a})=T_{\nu(z)} u^{sc}(z; D)\,e^{ik_p \textbf{a}\cdot d}.
\en
Recall that, the P-part and S-part far-field patterns of $u^\infty(\hat{x}; D_\textbf{a})$ can be respectively characterized as follows
(cf.~\cite{AK}):
\ben\no
 u^{\infty}_p(\hat{x}; D_\textbf{a})
&=&  \int_{\partial D_\textbf{a}}\big\{[T_{\nu(y)}\;\{  \hat{x}\otimes \hat{x}\, e^{-ik_p\hat{x}\cdot y}\}]^\top u^{sc}(y; D_\textbf{a})\\
&&- \hat{x}\otimes \hat{x}\, e^{-ik_p\hat{x}\cdot y}\;T_{\nu(y)} u^{sc}(y; D_\textbf{a})\big\}\; d  s(y),
\enn
\ben
u^{\infty}_s(\hat{x}; D_\textbf{a})
&=&  \int_{\partial D_\textbf{a}}\big\{ [T_{\nu(y)}\{(\textbf{I}-  \hat{x}\otimes \hat{x}) e^{-ik_s\hat{x}\cdot y}\}]^\top u^{sc}(y; D_\textbf{a})\\
&&-(\textbf{I}- \hat{x}\otimes \hat{x}) e^{-ik_s\hat{x}\cdot y} T_{\nu(y)} u^{sc}(y; D_\textbf{a})\big\} d  s(y).
\enn
Changing the variable $y=z+\textbf{a}$ in the previously two expressions and making use of (\ref{u1}) and (\ref{u2}), we obtain
\ben\no
 u^{\infty}_p(\hat{x}; D_\textbf{a})
&=&  \int_{\partial D}\big\{[T_{\nu(z)}\;\{  \hat{x}\otimes \hat{x}\, e^{-ik_p\hat{x}\cdot z}\}]^\top u^{sc}(z; D)\\
&&- \hat{x}\otimes \hat{x} e^{-ik_p\hat{x}\cdot z}\;T_{\nu(y)} u^{sc}(z; D)\big\}\; d  s(z)\;e^{ik_p(d-\hat{x})\cdot\textbf{a}}\\
&=&u^{\infty}_p(\hat{x}; D)\;e^{ik_p(d-\hat{x})\cdot\textbf{a}},
\enn
and
\ben
u^{\infty}_s(\hat{x}; D_\textbf{a})
&=&  \int_{\partial D}\big\{[T_{\nu(z)}\{(\textbf{I}-  \hat{x}\otimes \hat{x}) e^{-ik_s\hat{x}\cdot z}\}]^\top u^{sc}(z; D)\\
&&-(\textbf{I}- \hat{x}\otimes \hat{x}) e^{-ik_s\hat{x}\cdot z} T_{\nu(z)} u^{sc}(z; D)\big\} d  s(z)\;e^{i(k_pd-k_s \hat{x})\cdot\textbf{a}}\\
&=&u^{\infty}_s(\hat{x}; D)\;e^{i(k_pd-k_s \hat{x})\cdot\textbf{a}},
\enn
from which the relations in (\ref{Far-field-relation-p}) follow.
The case for incident plane shear waves can be treated in the same manner.
\end{proof}
\begin{rem}
For general plane waves of the form (\ref{plane}),
one can obtain the corresponding relations between translated elastic bodies by supposition, giving rise to the updating formulae in Step 9 of
 Scheme R.
 Note that the identities in
(\ref{Far-field-relation-p}) (resp. (\ref{Far-field-relation-s})) are valid for incident plane shear (resp. pressure) wave only.
\end{rem}
As an application of the relations established in Lemma \ref{translation-relation}, we prove uniqueness in locating the position of translated
elastic bodies with a single plane pressure or shear wave.
\begin{lem}\label{Lem:uniqueness}  Let $d, d^\bot\in\mathbb{S}^2$ and $\omega\in \R_+$ be fixed.
 Assume $\alpha\,\beta=0$. Then the relation $u^\infty_\tau(\hat{x}, D_\textbf{a})=u^\infty_\tau(\hat{x}, D)$ with $\tau=p$ or $\tau=s$
for all $\hat{x}\in\s^2$ implies $|\textbf{a}|=0$.
\end{lem}
\begin{proof} Without loss of generality we
assume $\beta=0$. This implies that the incident wave is a plane pressure wave.
If $u^\infty_p(\hat{x}; D_\textbf{a})=u^\infty_p(\hat{x}; D)$ for all $\hat{x}\in\s^2$,
then it follows from the first identity in (\ref{Far-field-relation-p}) that $(d-\hat{x})\cdot \textbf{a}=0$ for all
$\hat{x}\in\s^2$. Since the set $\{d-\hat{x}:\hat{x}\in\s^2 \}$ contains three linearly independent vectors of $\R^3$,
it follows that  $|\textbf{a}|=0$. By arguing similarly we deduce from $u^\infty_s(\hat{x}; D_\textbf{a})=u^\infty_s(\hat{x}; D)$
and the second identity in (\ref{Far-field-relation-p}) that $(k_pd-k_s\hat{x})\cdot \textbf{a}=0$ for all $\hat{x}\in\s^2$,
which also leads to $|\textbf{a}|=0$. The proof for the case with $\alpha=0$ can be shown in the same way by using (\ref{Far-field-relation-s}).
\end{proof}

Let $\mathcal{R}$ be a rotation matrix in $\R^3$. The following relation between
$u^{\infty}(\hat{x}; D)$ and $u^{\infty}(\hat{x}; \mathcal{R} D)$ was mentioned in \cite[Section 5]{MD}.
\begin{lem}\label{Lem:rotation}
\be\label{ball}
\mathcal{R}\, u^{\infty}(\hat{x};D, d, d^\bot)= u^{\infty}(\mathcal{R}\hat{x}; \mathcal{R}D, \mathcal{R}d,\mathcal{R}d^\bot),
\quad\mbox{for all}\quad \hat{x},d,d^\bot\in \mathbb{S}^2, d\cdot d^\bot=0.
\en
\end{lem}

 Now, we
recall the scaling operator
 $\Lambda_\rho D=\{\rho\,x: x\in D\}$ for $\rho>0$. Given the incident wave $u^{in}$ of the form (\ref{plane}), we write $u^\infty(\hat{x})=u^\infty(\hat{x}; D,  \omega)$ and $u^{sc}(\hat{x})=u^{sc}(\hat{x}; D,  \omega)$ to
indicate the dependance on the obstacle $D$ and the frequency of incidence $\omega$.
 There holds
\begin{lem}\label{Lem:scale}
\ben
u_\tau^\infty(\hat{x}; \Lambda_\rho D,  \omega )=\rho\,u_\tau^\infty (\hat{x}; D, \rho\, \omega ),\quad \hat{x}\in \s^2,\quad\tau=p,s\,\mbox{or}\;\emptyset.
\enn
\end{lem}
\begin{proof} From (\ref{far}), we see
\be\no
u^{sc}(x; \Lambda_\rho D,  \omega )&=&\frac{\exp(ik_p|x|)}{4\pi (\lambda+\mu)|x|}\, u^{\infty}_p(\hat{x}; \Lambda_\rho D,  \omega )+
\frac{\exp(ik_s|x|)}{4\pi\mu|x|}\, u^{\infty}_s(\hat{x}; \Lambda_\rho D,  \omega )\\ \label{far-sG}
&&+ \mathcal{O}(\frac{1}{|x|^2})
\en
as $|x|\rightarrow+\infty$. Define $\widetilde{u}^{sc}(y):=u^{sc}(\rho y; \Lambda_\rho D,  \omega  )$ for $y\in \R^3\backslash\overline{D}$. It is readily seen that
\ben\left\{\begin{array}{lll}
(\Delta^*+\rho^2\, \omega^2)\,\widetilde{u}^{sc}=0&& \mbox{in}\quad \R^3\backslash \overline{D},\\
\widetilde{u}^{sc}(y)=-u^{in}(y)&&\mbox{on}\quad \partial D.
\end{array}\right.
\enn
Moreover, $\widetilde{u}^{sc}(y)$ is still a radiating solution with the asymptotic behavior
\ben
\widetilde{u}^{sc}(y; D, \rho\omega )=\frac{\exp(i\rho k_p|y|)}{4\pi (\lambda+\mu)|y|}\, u^{\infty}_p(\hat{y}; D, \rho \omega )+
\frac{\exp(i\rho d_s|y|)}{4\pi \mu |y|}\, u^{\infty}_s(\hat{y}; D, \rho\omega )+ \mathcal{O}(\frac{1}{|y|^2}),
\enn as $|y|\rightarrow\infty$.
Changing the variable $y=x/\rho$, we deduce from the above expression that
\be\no
u^{sc}(x)=\widetilde{u}^{sc}(x/\rho)&=&\frac{\rho\exp(ik_p|y|)}{4\pi (\lambda+\mu) |x|}\, u^{\infty}_p(\hat{x}; D, \rho \omega )+
\frac{\rho\exp(ik_s|x|)}{4\pi \mu |x|}\, u^{\infty}_s(\hat{x}; D, \rho\omega )\\ \label{far-G}
&&+ \mathcal{O}(\frac{1}{|x|^2})
\en as $|x|\rightarrow\infty$.
Comparing (\ref{far-sG}) with (\ref{far-G}) yields
\ben
u^{\infty}_\tau(\hat{x}; \Lambda_\rho D,  \omega )=\rho\,u^{\infty}_\tau(\hat{x}; D, \rho \omega ),\quad\tau=p,s,
\enn
and thus
\ben
u^{\infty}(\hat{x}; \Lambda_\rho D,  \omega )=\rho\,u^{\infty}(\hat{x}; D, \rho \omega ).
\enn

The proof is completed.
\end{proof}

We are in a position to present the proof of Theorem \ref{Th2}.

\begin{proof}[Proof of Theorem \ref{Th2}] Without loss of generality, we
assume $\alpha=1, \beta=0$. Let the scatterer component $\Om_j=\Om_j(z_j, \mathcal{R}_j, r_j, E_j)$ fulfill (\ref{Omega}) and (\ref{eq:b2}).
Combining Lemmas \ref{translation-relation}, \ref{Lem:rotation} and \ref{Lem:scale}, we obtain
\be\label{eq12}
u_p^\infty(\hat{x};  d, \omega, \Omega_j)&=&u_p^\infty(\hat{x}; d, \omega,  \mathcal{R}_j\, \Lambda_{r_j}\, E_j)\,e^{ik_p(d-\hat{x})\cdot z_j}\\ \no
&=&\mathcal{R}_j\, u_p^\infty(\mathcal{R}_j ^{-1}\hat{x};  \mathcal{R}_j ^{-1}d,\, \omega,\,  \Lambda_{r_j}\, E_j)\,e^{ik_p(d-\hat{x})\cdot z_j}\\ \label{eq9}
&=&r_j\,\mathcal{R}_j\, u_p^\infty(\mathcal{R}_j ^{-1}\hat{x};  \mathcal{R}_j ^{-1}d,\, r_j\,\omega,\,  E_j)\,e^{ik_p(d-\hat{x})\cdot z_j}.
\en
Using (\ref{eq:b2}) and again Lemmas \ref{Lem:rotation} and \ref{Lem:scale}, we have for $j\in J_0$
\be\no
&&r_j\,\mathcal{R}_j\, u_p^\infty(\mathcal{R}_j ^{-1}\hat{x};  \mathcal{R}_j ^{-1}d,\, r_j\,\omega,\,  E_j)\\ \no
&=& r_{j_{\tau}}\,\mathcal{R}_{j_{\sigma}}\, u_p^\infty(\mathcal{R}_{j_{\sigma}}^{-1}\hat{x}; \mathcal{R}_{j_{\sigma}}^{-1} d, r_{j_{\tau}}\,\omega,\, \, \Sigma_{j_0}  )+\mathcal{O}(h)\\ \no
&=& u_p^\infty(\hat{x};  d, \omega, \mathcal{R}_{j_{\sigma}} \Lambda_{r_{j_{\tau}}} \Sigma_{j_0},)+\mathcal{O}(h)\\ \label{eq8}
&=& u_p^\infty(\hat{x};  d, \omega,  \widetilde{\Sigma}_1)+\mathcal{O}(h),
\en
where $\widetilde{\Sigma}_1$ is given as in (\ref{eq:b1}).
Inserting (\ref{eq8}) into (\ref{eq9}), it follows from Lemma \ref{lem1} that
\ben\no
u_p^\infty(\hat{x}; \Omega)&=&\sum_{j=1}^{l_e} u_p^\infty(\hat{x}; \Omega_j) +\mathcal{O}(L_e^{-1})\\
&=&\sum_{j\in J_0} u_p^\infty(\hat{x}; \widetilde{\Sigma}_1)e^{ik_p(d-\hat{x})\cdot z_j}+\sum_{j\in \{1,\cdots l_e\}\backslash J_0} u_p^\infty(\hat{x}; \Omega_j) +\mathcal{O}(L^{-1}_e+h).
\enn
Hence, for $z\in {\rm neigh}(z_{j})$ with some $j\in J_0$ we have
\be\no
&&|\langle u_p^\infty(\hat{x}; \Omega),  u_p^\infty(\hat{x}; \widetilde{\Sigma}_1) e^{-ik_p\hat{x}\cdot z} \rangle|\\ \label{eq10}
&=&|\langle u_p^\infty(\hat{x}; \widetilde{\Sigma}_1)e^{ik_p(d-\hat{x})\cdot z_{j}},  u_p^\infty(\hat{x}; \widetilde{\Sigma}_1) e^{-ik_p\hat{x}\cdot z} \rangle|+\mathcal{O}(L_e^{-1}+h)\\ \label{eq11}
&\leq & \|u_p^\infty(\hat{x}; \widetilde{\Sigma}_1)\|_{L^2}+\mathcal{O}(L^{-1}_e+h).
\en
The equality in (\ref{eq10}) follows from the the Riemann-Lebesgue lemma about
oscillatory integrals by noting $|z_{j'}-z|\sim L_e\gg 1$ for $j'\neq j, 1\leq j'\leq l_e$ and $z\in {\rm neigh}(z_{j})$.
For the inequality in (\ref{eq11}), we have applied the Cauchy-Schwartz inequality,
and it is easily seen that the equality holds only at $z = z_{j}$. Therefore, from the definition of the indicator function $W_1^1$,
\ben
W_1^1(z)\leq 1+\mathcal{O}(L^{-1}_e+h),\quad\mbox{for}\quad z\in {\rm neigh}(z_{j}).
\enn
On the other hand, by a similar argument, together with
assumption (i) on $\widetilde{\Sigma}_j$ and the equality (\ref{eq12}) , we can directly verify that
\ben
W_1^1(z)<1+\mathcal{O}(L^{-1}_e+h),\quad z\in {\rm neigh}(z_{j}),\quad j\in\{1,2,\cdots,l_e\}\backslash J_0.
\enn
This proves Theorem \ref{Th2} with $m=1$ for an incident pressure wave. In a completely similar manner, our argument can be extended to show the indicating behavior of $W_2^1(z)$ ($m=2$) by using the first equality in (\ref{Far-field-relation-s}). Regarding $W_3^1(z)$ ($m=3$) where the full far-field pattern data are involved, we apply the orthogonality of $u_p^\infty$ and $u_s^\infty$ to obtain
\ben
W_3^j=\frac{\left|\langle  u_p^\infty(\hat{x}; \Om),\; e^{-ik_p\hat{x}\cdot z}\,u_p^\infty(\hat{x}; \widetilde{\Sigma}_j)\rangle+
\langle u_s^\infty(\hat{x}; \Om), e^{-ik_s\hat{x}\cdot z}\,u_s^\infty(\hat{x}; \widetilde{\Sigma}_j) \rangle \right|^2}{\|u_p^\infty(\hat{x}; \widetilde{\Sigma}_j)\|^2_{L^2}+ \|u_s^\infty(\hat{x}; \widetilde{\Sigma}_j)\|^2_{L^2} }\,.
\enn
Thus, the behavior of $W_3^1(z)$ follows from those of $W_1^1(z)$ and $W_2^1(z)$.

In the case of an incident shear wave, the indicating behavior of $W_m^1(z)$ ($m=1,2,3$)
can be shown similarly . The proof of Theorem \ref{Th2} is complete.
\end{proof}

\begin{rem}\label{Re:Ex} For a general incident plane wave of the form (\ref{plane}), following a similar argument to the proof of Theorem~\ref{Th2}, one can show that Theorem~\ref{Th2} still holds with the indicator functions replaced, respectively, by
\ben
W_1^j(z)=\frac{\left|\langle  u_p^\infty(\hat{x}; \Om),\;A_1^j(\hat{x};z)\rangle \right|^2}
{\|u_p^\infty(\hat{x}; \widetilde{\Sigma}_j)\|^2_{L^2}}, \quad
W_2^j(z)=\frac{\left|\langle  u_s^\infty(\hat{x}; \Om),\;A_2^j(\hat{x};z)\rangle \right|^2}
{\|u_s^\infty(\hat{x};  \widetilde{\Sigma}_j)\|^2_{L^2}},\\
W_3^j(z)=\frac{\left|\langle  u^\infty(\hat{x}; \Om),\;A_1^j(\hat{x};z)+A_2^j(\hat{x};z)\rangle \right|^2}
{\|u^\infty(\hat{x};  \widetilde{\Sigma}_j)\|^2_{L^2}},
\enn
where for $j=1,2,\cdots,l''$,
\ben
A_1^j(\hat{x};z)&:=&e^{ik_p(d-\hat{x})\cdot z}\,u_p^\infty(\hat{x}; d, d^\bot, \alpha, 0, \widetilde{\Sigma}_j)+ e^{i(k_sd-k_p\hat{x})\cdot z}\,u_p^\infty(\hat{x}; d, d^\bot, 0, \beta, \widetilde{\Sigma}_j),\\
A_2^j(\hat{x};z)&:=& e^{i(k_pd-k_s\hat{x})\cdot z}\,u_s^\infty(\hat{x}; d, d^\bot, \alpha, 0, \widetilde{\Sigma}_j)+ e^{ik_s(d-\hat{x})\cdot z}\,u_s^\infty(\hat{x}; d, d^\bot, 0, \beta, \widetilde{\Sigma}_j).
\enn
\end{rem}

\section{Locating multiple multiscale scatterers}\label{section-multi-scale}

In this section, we consider the recovery of a scatterer consisting of multiple multiscale components given by
\begin{equation}\label{eq:multiscale}
G=D\cup \Omega,
\end{equation}
where $D$ is as described in \eqref{scaling}--\eqref{assumption-small} denoting the union of the small components, and $\Omega$ is as described in \eqref{Omega}--\eqref{eq:admss} denoting the union of the extended components. As before, we assume that the shapes of the extended components are from a known admissible class, as described in \eqref{E0}--\eqref{eq:admss}. In addition, we require that
\begin{equation}\label{eq:sparsem}
L_m:={\rm dist}(D, \Omega)\gg 1.
\end{equation}
Next, we shall develop Scheme M to locate the $l_s+l_e$ multiscale scatterer components of $G$ in \eqref{eq:multiscale} by using a single far-field pattern. Our treatment shall follow the one in \cite{LLZ} of locating multiscale acoustic scatterers. More specifically, we shall concatenate Schemes S and R of locating small and extended scatterers, respectively, by a {\it local tuning} technique, to form Scheme M of locating the multiscale scatterers.

\begin{defn}\label{def:1}
Let $\mathscr{A}_h$ be the augmented admissible class in \eqref{eq:a3} with the two sets $\{\mathcal{R}_j\}_{j\in\mathscr{N}_1}$ and $\{r_j\}_{j\in\mathscr{N}_2}$ of rotations and scalings respectively, and $\mathcal{T}_h$ be the sampling mesh in Scheme R. Suppose that $\widehat{\Omega}_j=\widehat{z}_j+\widehat{\mathcal{R}}_j\Lambda_{\widehat{r}_j} \Sigma_j$, $j=1,2,\ldots,l_e$, are the reconstructed images of $\Omega_j=z_j+\mathcal{R}_j\Lambda_{r_j}\Sigma_j$, $j=1,2,\ldots,l_e$. For a properly chosen $\delta\in\mathbb{R}_+$, let $\mathscr{O}_1^j, \mathscr{O}_2^j$ and $\mathscr{O}_3^j$ be, respectively, $\delta$-neighborhoods of $\widehat{z}_j$, $\widehat{\mathcal{R}}_j$ and $\widehat{r}_j$, $j=1,2,\ldots,l_e$. Then let $\{\mathcal{T}_{h_l'}, \{\mathcal{R}_j\}_{j\in\mathscr{P}_l}, \{r_j\}_{j\in\mathscr{Q}_l}\}$  be a refined mesh of $\{\mathcal{T}_h\cap\mathscr{O}_1^l, \{\mathcal{R}_j\}_{j\in\mathscr{N}_1}\cap\mathscr{O}_2^l, \{r_j\}_{j\in\mathscr{N}_2}\cap\mathscr{O}_3^l\}$, $l=1,2,\ldots,l_e$.

Define
\begin{equation}\label{eq:localtp}
\widehat{\widehat{\Omega}}_l(\widehat{\widehat{z}},\widehat{\widehat{\mathcal{R}}},\widehat{\widehat{r}}):=\widehat{\widehat{z}}+\widehat{ \widehat{ \mathcal{R} } } \Lambda_{\widehat{\widehat{r}}}\Sigma_l\quad \mbox{for}\ \ \widehat{\widehat{z}}\in \mathcal{T}_{h_l'},\ \ \widehat{\widehat{\mathcal{R}}}\in\{\mathcal{R}_j\}_{j\in\mathscr{P}_l}, \ \ \widehat{\widehat{r}}\in\{r_j\}_{j\in\mathscr{Q}_l},
\end{equation}
a {\it local tune-up} of $\widehat{\Omega}_l=\widehat{z}_l+\widehat{\mathcal{R}}_l\Lambda_{\widehat{r}_l} \Sigma_l$ relative to $\{\mathcal{T}_{h_l'}, \{\mathcal{R}_j\}_{j\in\mathscr{P}_l}, \{r_j\}_{j\in\mathscr{Q}_l}\}$, $1\leq l\leq l_e$.

Define
\begin{equation}\label{eq:ltu}
\widehat{\widehat{\Omega}}:=\bigcup_{l=1}^{l_e} \widehat{\widehat{\Omega}}_l,
\end{equation}
with each $\widehat{\widehat{\Omega}}_l$, $1\leq l\leq l_e$, a local tune-up in \eqref{eq:localtp} relative to $\{\mathcal{T}_{h_l'}, \{\mathcal{R}_j\}_{j\in\mathscr{P}_l}, \{r_j\}_{j\in\mathscr{Q}_l}\}$, a {\it local tune-up} of $\widehat{\Omega}:=\bigcup_{j=1}^{l_e}\widehat{\Omega}_j$, relative to the {\it local tuning mesh}
\begin{equation}\label{eq:mesh}
\mathscr{L}:=\bigcup_{l=1}^{l_e} \{\mathcal{T}_{h_l'}, \{\mathcal{R}_j\}_{j\in\mathscr{P}_l}, \{r_j\}_{j\in\mathscr{Q}_l}\}.
\end{equation}
\end{defn}

According to Definition~\ref{def:1}, $\widehat{\Omega}$ is the reconstructed image of the extended scatterer $\Omega$, whereas $\widehat{\widehat{\Omega}}$ is an adjustment of $\widehat{\Omega}$ by locally adjusting the position, orientation and size of each component of $\widehat{\Omega}$.

With the above preparation, we are ready to present Scheme M to locate the multiple components of $G$ in \eqref{eq:multiscale}, which can be first sketched as follows. First, by Lemmas~\ref{lem1} and \ref{lem2}, we know
\begin{equation}\label{eq:m1}
u_\tau^\infty(\hat x; G)\approx u_\tau(\hat x; \Omega), \quad \tau=s,p\ \mbox{or}\ \emptyset,
\end{equation}
where $u_\tau^\infty(\hat x; G)$ and $u_\tau^\infty(\hat x; \Omega)$ are, respectively, the far-field patterns of $G$ and $\Omega$ corresponding to a single incident plane wave of the form (\ref{plane}). Hence, one can use $u_\tau^\infty(\hat x; G)$ as the far-field data for Scheme R to locate the extended scatterer components of $\Omega$ (approximately). We suppose the reconstruction in the above step yields $\widehat{\Omega}$, which is an approximation to $\Omega$. Then, according to Lemma~\ref{lem1} again, we have
\begin{equation}\label{eq:m2}
u_\tau(\hat x; D)\approx u_\tau^\infty(\hat x;G)-u_\tau^\infty(\hat x; \Omega)\approx u_\tau^\infty(\hat x; G)-u_\tau^\infty(\hat x; \widehat{\Omega}). 
\end{equation}
With the above calculated far-field data, one can then use Scheme S to locate the small scatterer components of $D$. However, the error introduced in \eqref{eq:m2} might be even more significant than the scattering data of $D$, hence the second-stage reconstruction cannot be expected to yield some reasonable result. In order to tackle this problem, a {\it local tuning technique} can be implemented by replacing $\widehat{\Omega}$ in \eqref{eq:m2} by a local tune-up $\widehat{\widehat{\Omega}}$. Clearly, a more accurate recovery of the extended scatterer $\Omega$ is included in the local tune-ups relative to a properly chosen local tuning mesh. Hence, one can repeat the second-stage reconstruction as described above by running through all the local tune-ups, and then locate the ``clustered'' local maximum points which denote the positions of the small scatterers. Meanwhile, one can also achieve much more accurate reconstruction of the extended scatterers. In summary, Scheme M can be proceeded as follows.

 \vspace*{3mm}

\begin{tabularx}{\textwidth }{>{\bfseries}lX}
\toprule
Scheme M&  Locating multiple multi-scale scatterers of $G$ in \eqref{eq:multiscale}.
\\ \midrule
Step 1 & For an unknown scatterer $G$,
  collect the P-part ($u_p^\infty (\hat x; G)$), S-part ($u_s^\infty(\hat x; G)$) or the full far-field ($u^\infty(\hat x; G)$) patterns, by sending a single detecting plane wave of the form (\ref{plane}).  \\ \midrule
Step 2 & Select a sampling region with a mesh $\mathcal{T}_h$ containing $\Omega$. \\ \midrule
Step 3 & Apply Scheme M with $u_\tau^\infty(\hat x; G)$, $\tau=s, p$ or $\emptyset$, as the far-field data, to reconstruct approximately the extended scatterer $\Omega$, denoted by $\widehat{\Omega}$. Clearly, $\widehat{\Omega}$ is as described in Definition~\ref{def:1}.  \\ \midrule
Step 4 & For $\widehat{\Omega}$ obtained in {\bf Step 3}, select a local-tuning mesh $\mathscr{L}$ of the form \eqref{eq:mesh}.   \\ \midrule
Step 5 & For a tune-up $\widehat{\widehat{\Omega}}$ relative to the local tuning mesh $\mathscr{L}$ in {\bf Step 4}, calculate
\begin{equation}\label{eq:formu}
\widetilde{u}_\tau^\infty(\hat x):=u_\tau^\infty(\hat x; G)-u_\tau^\infty(\hat x; \widehat{\widehat{\Omega}}).
\end{equation}
Apply Scheme S with $\widetilde{u}_\tau^\infty(\hat x)$ as the far-field data to locate the significant local maximum points on $\mathcal{T}_h\backslash\mathscr{L}$.
 \\ \midrule
Step 6 & Repeat {\bf Step 5} by running through all the local tune-ups relative to $\mathscr{L}$. Locate the clustered local maximum points on $\mathcal{T}_h\backslash\mathscr{L}$, which correspond to the small scatterer components of $D$. \\ \midrule
Step 7 & Update $\widehat{\Omega}$ to the local tune-up $\widehat{\widehat{\Omega}}$ which generates the clustered local maximum points in {\bf Step 6}. \\
\bottomrule
\end{tabularx}
%
%

\section{Numerical examples}\label{Numerics}

In this section, three numerical tests are presented to verify the applicability of the proposed  new schemes (Scheme S, R and M) in inverse elastic scattering problems for rigid bodies in three dimensions. Either plane pressure wave or shear wave can be used as the detecting field incident on the rigid scatterer and it generates coexisting scattering P- and S-waves coupled by the rigid body boundary condition. However, for brevity, we only present the numerical results where the plane shear wave is employed for the locating schemes.

In the sequel, the exact far field data are synthesized by a forward solver using quadratic finite elements for each displacement field component on a truncated  spherical (3D) domain centered at the origin and enclosed by a PML layer following \cite{BPT}. The computation is carried out on a sequence of successively refined meshes till the relative error of two successive finite element solutions between the two adjacent meshes is below $0.1\%$.  The synthetic  far field data are computed via the integral representation formulae \cite[Eqs.~(2.12) and (2.13) ]{AK} and taken as the exact one.

In all the experiments, we always take the Lam{\' e} constants $\lambda=2$ and $\mu=1$, the incident direction  $d=(0,\,0,\,1)$, the perpendicular direction $d^\perp=(1,\,0,\,0)$ and the angular frequency $\omega=2$. In such a way, we know that the two wavenumbers $k_{p}=1$ and $k_{s}=2$ and the incident S-wavelength is $\pi$. For scatterers of  small size or regular size, we always add to the exact  far field data a uniform noise of $5\%$  and use it as the measurement data in our numerical tests. While for multiscale scatterers, a uniform noise of $3\%$ is added to the exact far field data.

Five revolving bodies will be considered for the scatterer components in our numerical tests. They are characterized by revolving  the following 2D parametric  curves along the $x$-axis. Some geometries are adjusted to their upright positions if necessary.
\begin{eqnarray*}
\mathbf{Ball:} & & \{ (x,y) : x=\cos(s), \  y=\sin(s), \  0\le s\le 2\pi \},\\
\mathbf{Peanut:} & & \{ (x,y) : x=\sqrt{3 \cos^2 (s) + 1}\cos(s), \  y=\sqrt{3 \cos^2 (s) + 1}\sin(s), \  0\le s\le 2\pi \},\\
\mathbf{Kite:} & & \{ (x,y) : x=\cos(s)+0.65\cos (2s)-0.65, \  y=1.5\sin(s), \  0\le s\le 2\pi \},\\
\mathbf{Acorn:} & & \{ (x,y) : x=(1+\cos(\pi s)\cos(2\pi s)/3)\cos(\pi s), \\
          & &  \phantom{ \{ (x,y) : }\  y=(1+\cos(\pi s)\cos(2\pi s)/3)\sin(\pi s), \  0\le s\le 2\pi \},\\
\mathbf{UFO:} & & \{ (x,y) : x=(1+0.2\cos(4\pi s))\cos(\pi s), \\
          & &  \phantom{ \{ (x,y) : }\ y=(1+0.2\cos(4\pi s))\sin(\pi s), \  0\le s\le 2\pi \}.
\end{eqnarray*}
They will be denoted for short by \textbf{B}, \textbf{P}, \textbf{K}, \textbf{A} and \textbf{U}, respectively, and shown in Figs.~\ref{fig:Ex1_true}(b), \ref{fig:Ex1_true}(c), \ref{fig:Ex1_true}(d) and Figs.~\ref{fig:Ex3_true}(b), \ref{fig:Ex3_true}(c).

\medskip

\textbf{Example 1}. (Scatterer of three small components) The scatterer consists of three
components \textbf{B}, \textbf{P} and \textbf{K}, all of which are  scaled by one tenth so that their sizes are  much smaller than the incident wave length.  As shown in Fig.~\ref{fig:Ex1_true}(a), one small ball is located at $(-2,\,3,\,-2)$, a small
peanut at $(3,\,-2,\,-2)$ and a small kite at $(3,\,3,\,3)$.   With resort to Scheme S, the reconstruction results of the small components are shown in Fig. \ref{fig:Ex1_result}  based on the indicator functions $I_1(z)$, $I_2(z)$ and $I_3(z)$  using the P-wave, S-wave and full-wave far field data, respectively.  It is clearly seen from Fig. \ref{fig:Ex1_result} that all the indicator functions $I_m$ ($m=1,2$) in Scheme S can identify the scatterer with the correct positions of its three  components. As emphasized in Remark~\ref{rem:22} that the resolution of the S-wave reconstruction in Fig.~\ref{fig:Ex1_result}(b) is much sharper than its P-wave counterpart  in Fig.~\ref{fig:Ex1_result}(a) due to the shorter wavelength of S-wave. However, the full-wave imaging result in Fig.~\ref{fig:Ex1_result}(c) exhibits the most accurate and stable reconstruction compared with the other two in that $I_3(z)$, by combining  the highlighted ball and kite positions (lower two components in Fig.~\ref{fig:Ex1_result}(a)) from $I_1$ and the highlighted peanut position (upper component in Fig.~\ref{fig:Ex1_result}(b)) from $I_1$, yields the best indicating behavior and meanwhile retains the resolution as in the S-wave scenario. Thus to avoid redundancy and for better resolution, we always take the full-wave indicator function for later examples.

\begin{figure}
\hfill{}\includegraphics[width=0.24\textwidth]{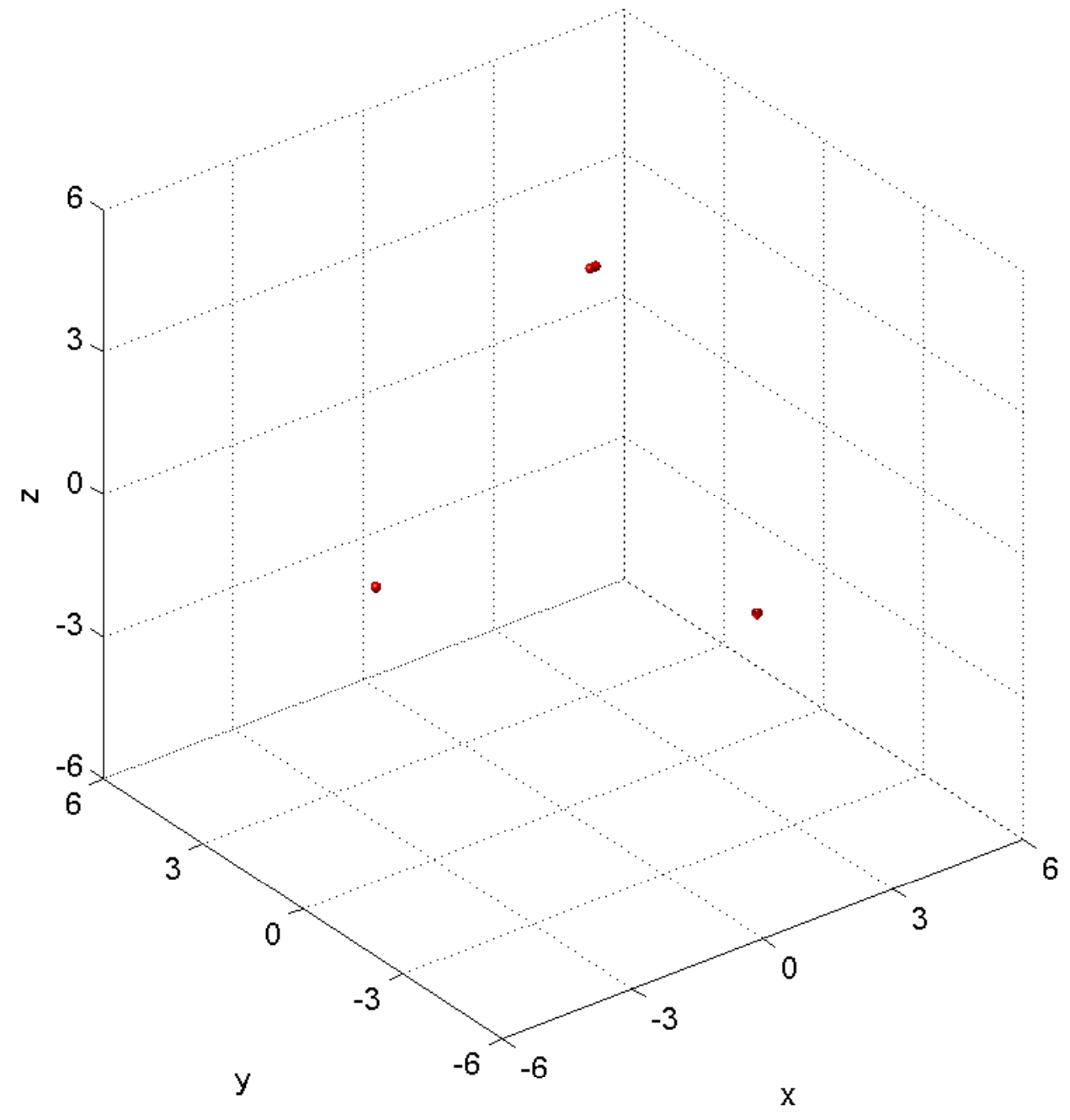}\hfill{}\includegraphics[width=0.24\textwidth]{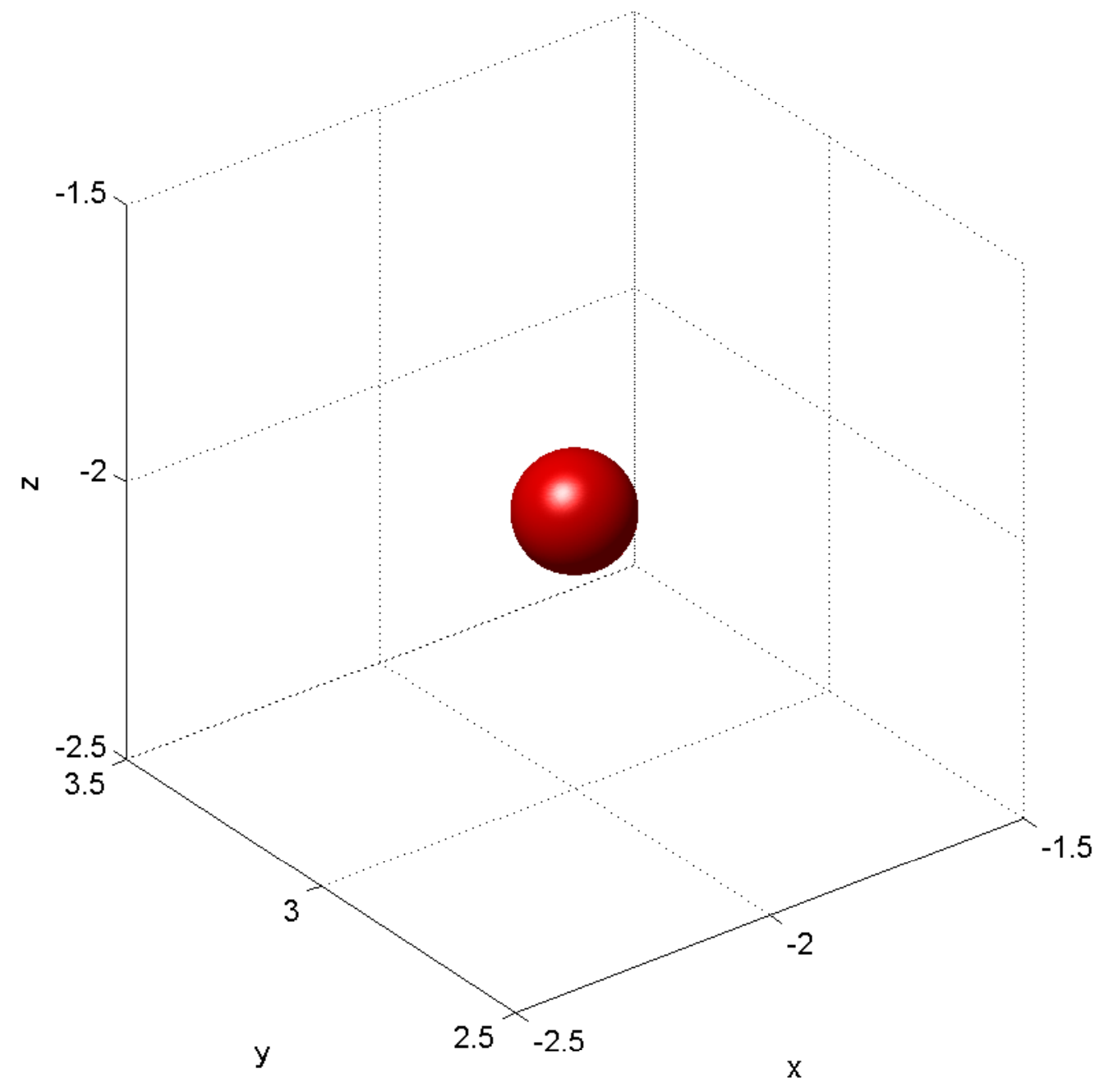}\hfill{}\includegraphics[width=0.24\textwidth]{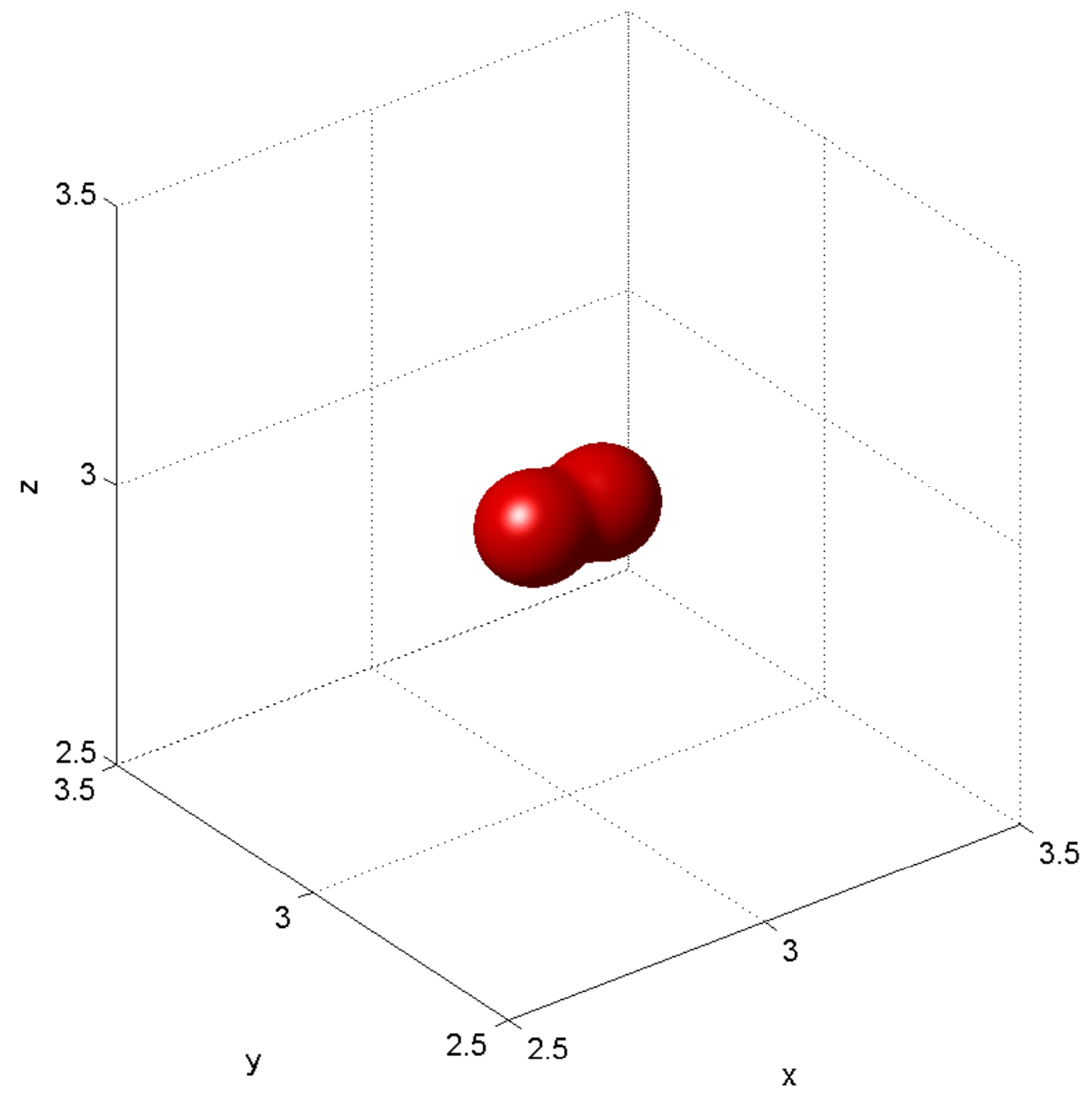}\hfill{}\includegraphics[width=0.24\textwidth]{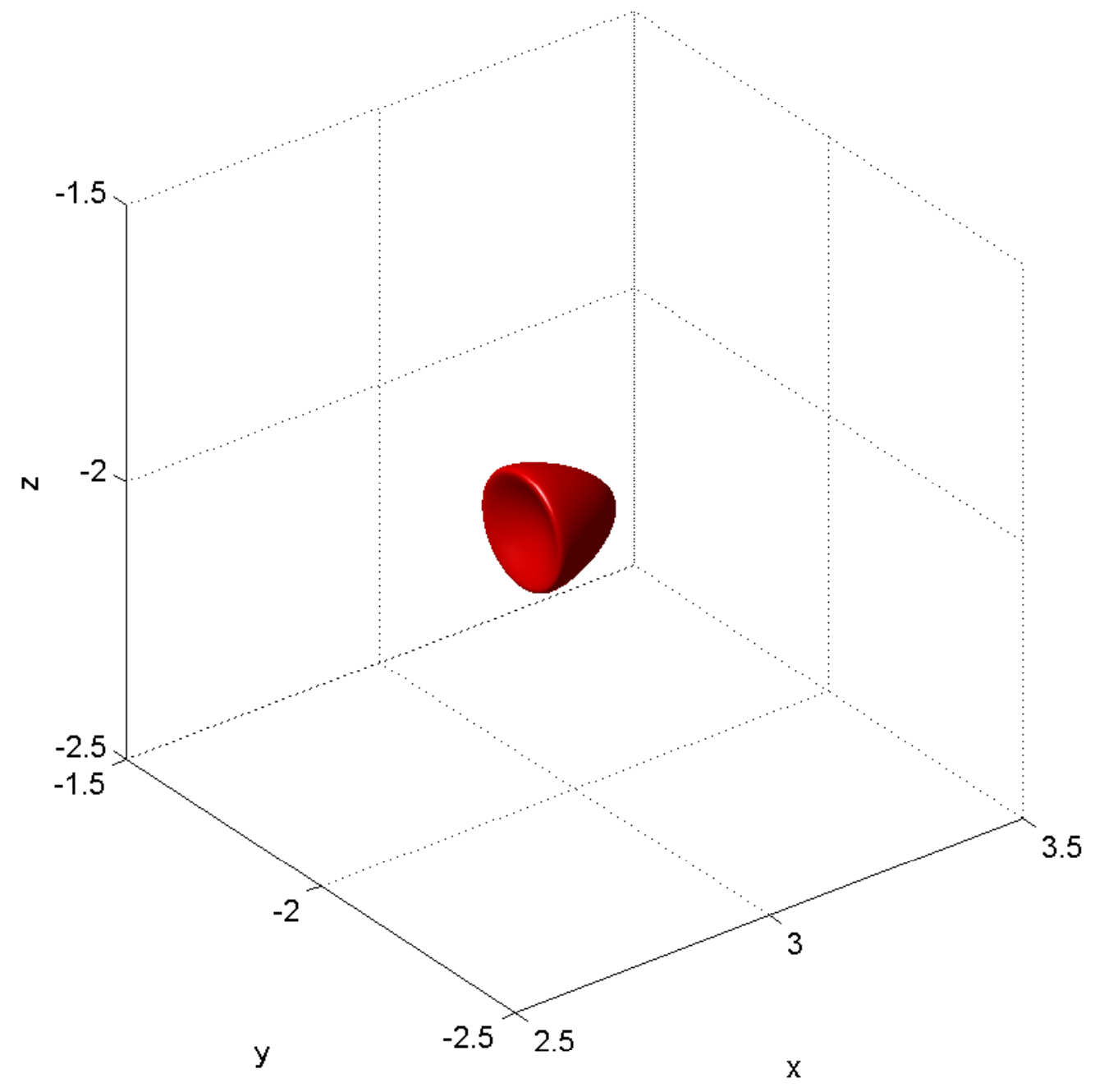}\hfill{}

\hfill{}(a)\hfill{}\hfill{}(b)\hfill{}\hfill{}(c)\hfill{}\hfill{}(d)\hfill{}

\caption{\label{fig:Ex1_true} True scatter and its  components before scaling in Example 1.}
\end{figure}

\begin{figure}
\hfill{}\includegraphics[width=0.32\textwidth]{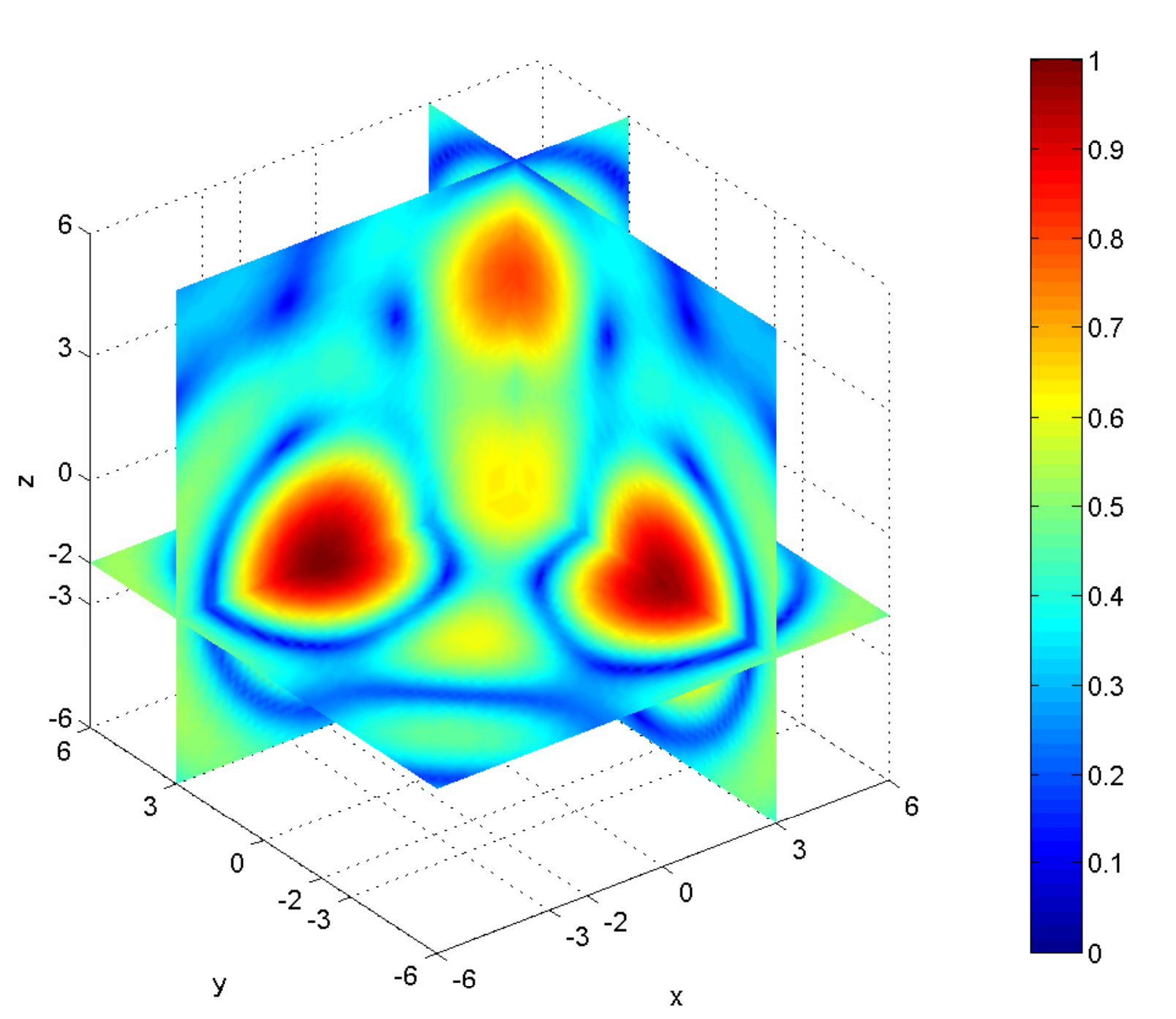}\hfill{}\includegraphics[width=0.32\textwidth]{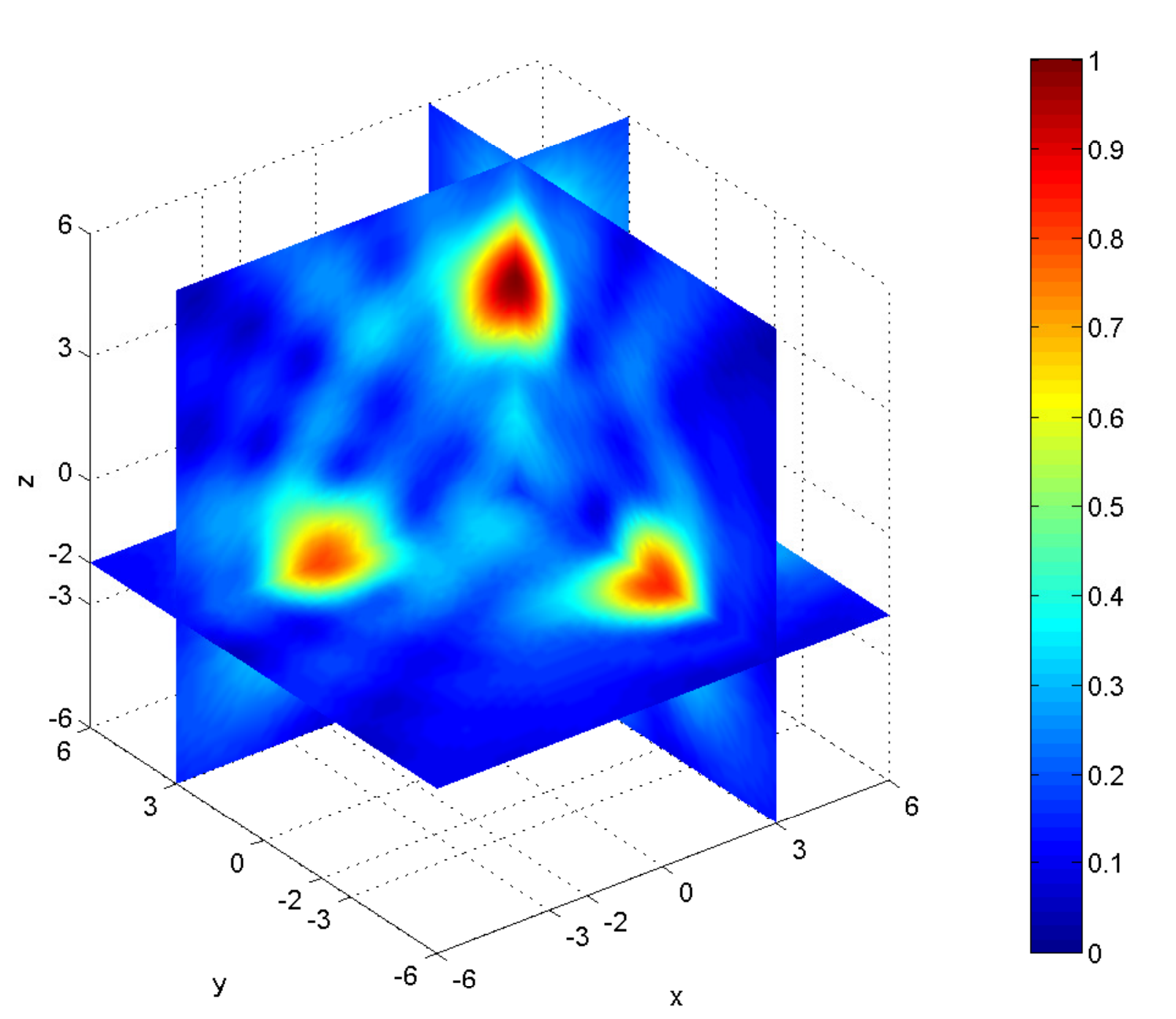}\hfill{}\includegraphics[width=0.32\textwidth]{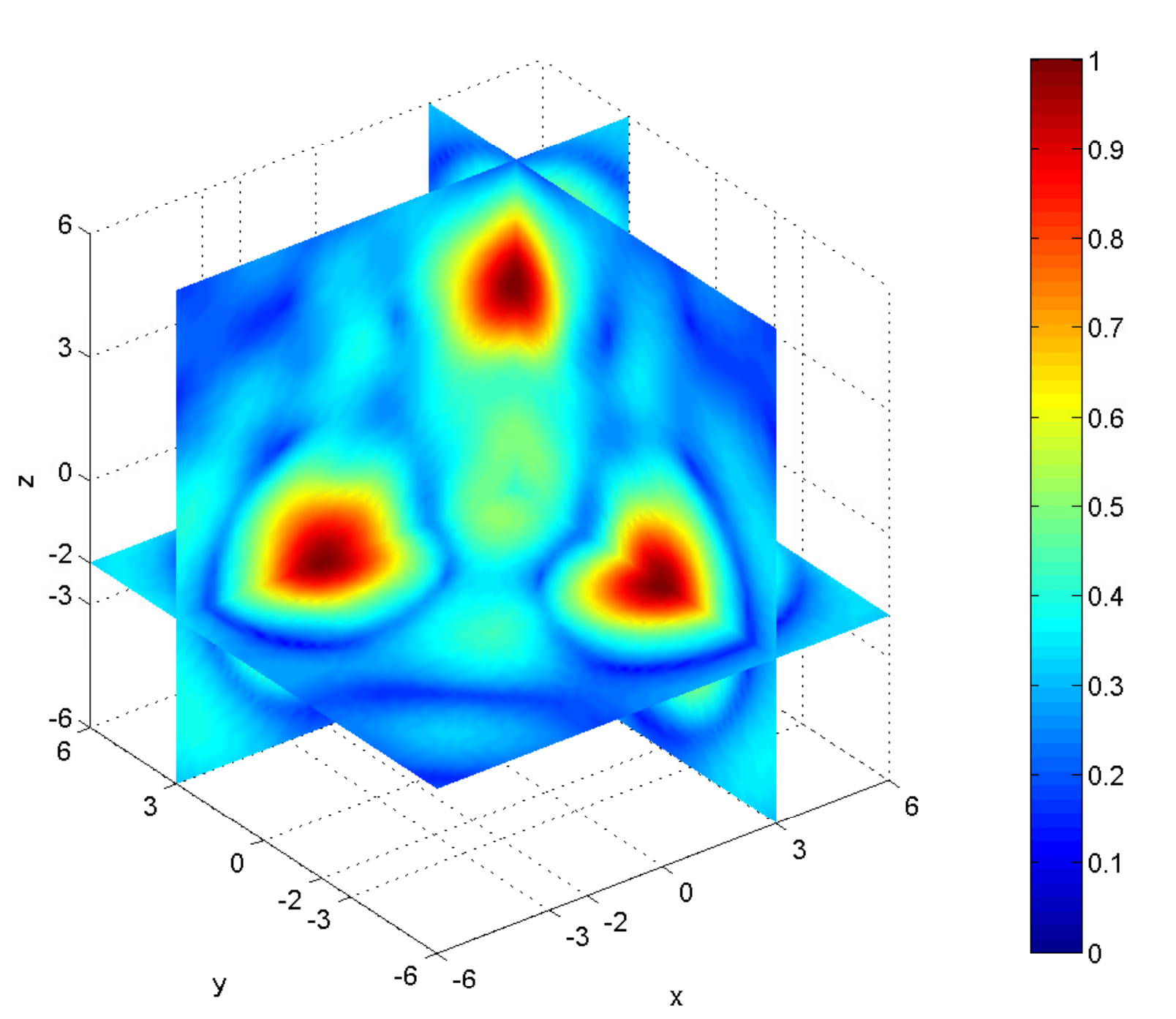}\hfill{}

\hfill{}P-wave\hfill{}\hfill{}S-wave\hfill{}\hfill{}Full-wave\hfill{}

\caption{\label{fig:Ex1_result} From left to right :  Reconstruction results based on the indicator functions $I_1(z)$, $I_2(z)$ and $I_3(z)$ using $u_p^\infty$, $u_s^\infty$ and $u^\infty$, respectively, in Example 1. }
\end{figure}

\medskip

\textbf{Example 2}. (Multiple extended scatterers) The scatterer is composed of a
UFO and an acorn. Their sizes are around 3 and is comparable with the incident plane shear wave.
The  UFO is located at $(-2,\,0,\,-2)$, and
the acorn is located at $(2,\,0,\,2)$ as shown in  Fig. \ref{fig:Ex3_true}(a).

\begin{figure}
\hfill{}\includegraphics[width=0.32\textwidth]{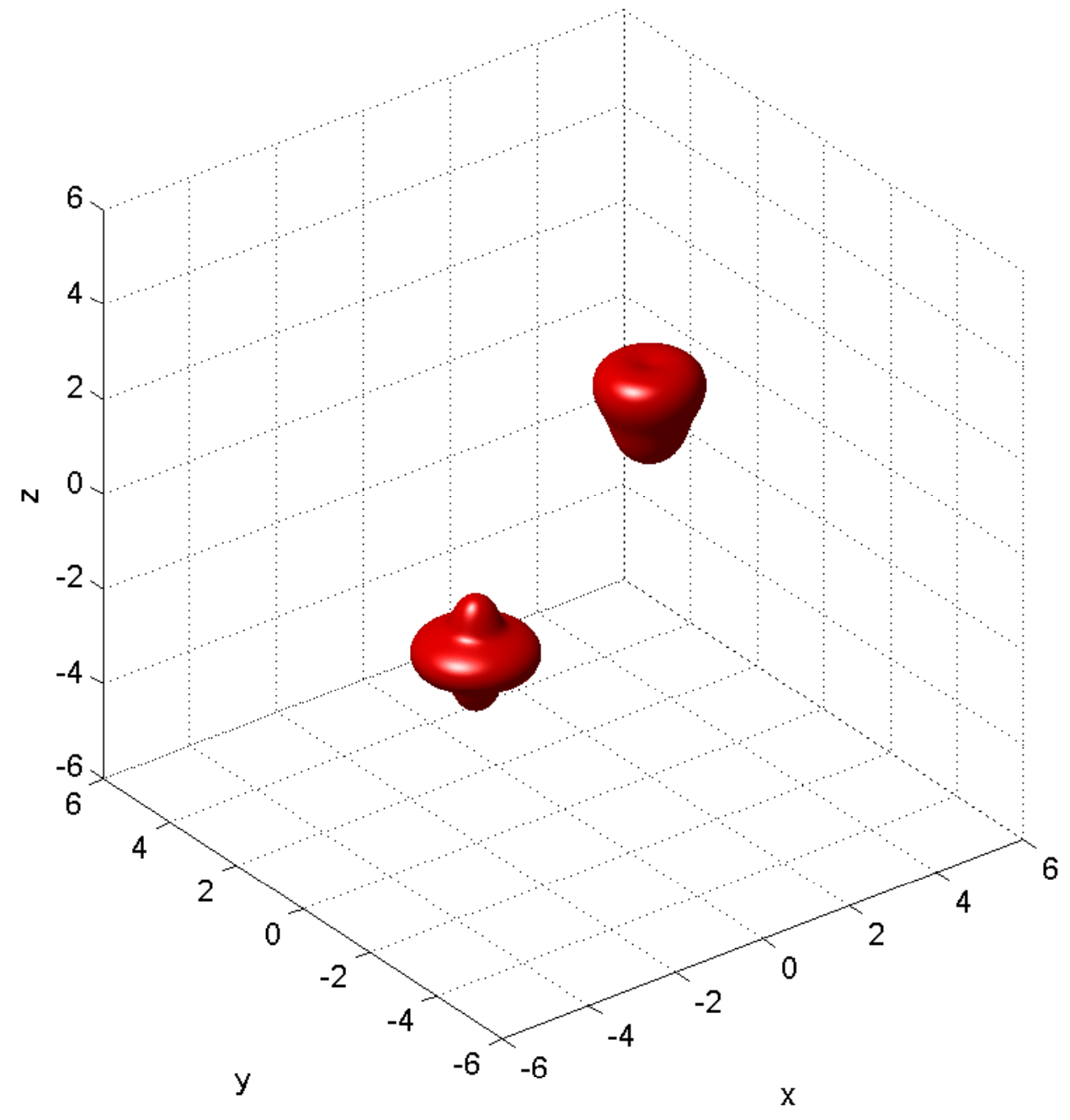}\hfill{}\includegraphics[width=0.32\textwidth]{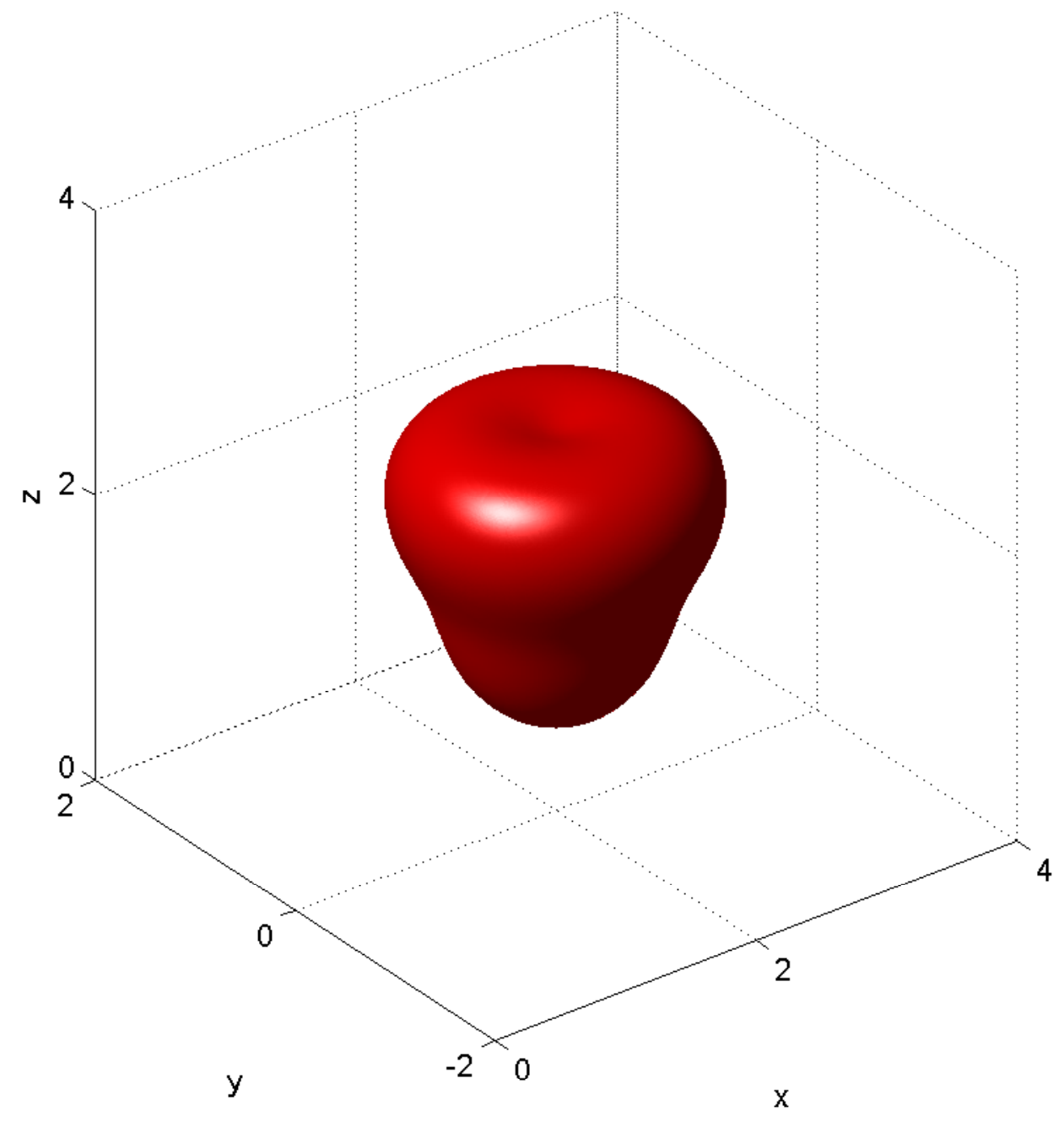}\hfill{}\includegraphics[width=0.32\textwidth]{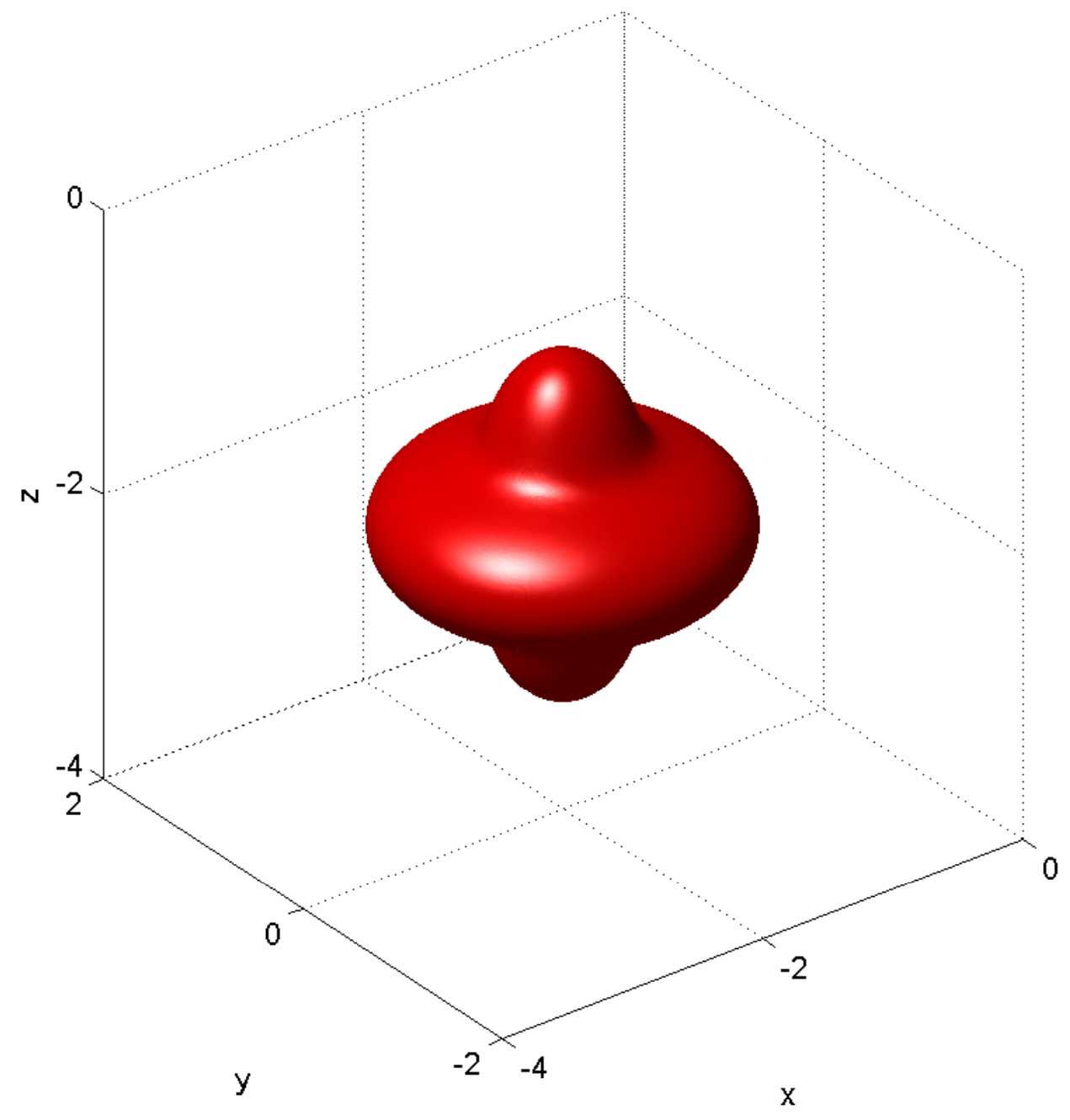}\hfill{}

\hfill{}(a)\hfill{}\hfill{}(b)\hfill{}\hfill{}(c)\hfill{}

\caption{\label{fig:Ex3_true}True scatter and its  components in Example 2.}
\end{figure}

The candidate  data set ${\mathscr{A}}_h$  includes  far-field data of both reference components  \textbf{U} and \textbf{A}, and is further lexicographically augmented by  a collection of a priori known orientations and sizes. More precisely, the augmented data set is obtained by rotating \textbf{U} and \textbf{A} in the  $x$-$z$ plane every 90 degrees, see, e.g., the four orientations of \textbf{A}  in Fig.~\ref{fig:Ex3_basis}, and by scaling \textbf{U} and \textbf{A} by  one half, one and twice.

\begin{figure}
\hfill{}\includegraphics[width=0.24\textwidth]{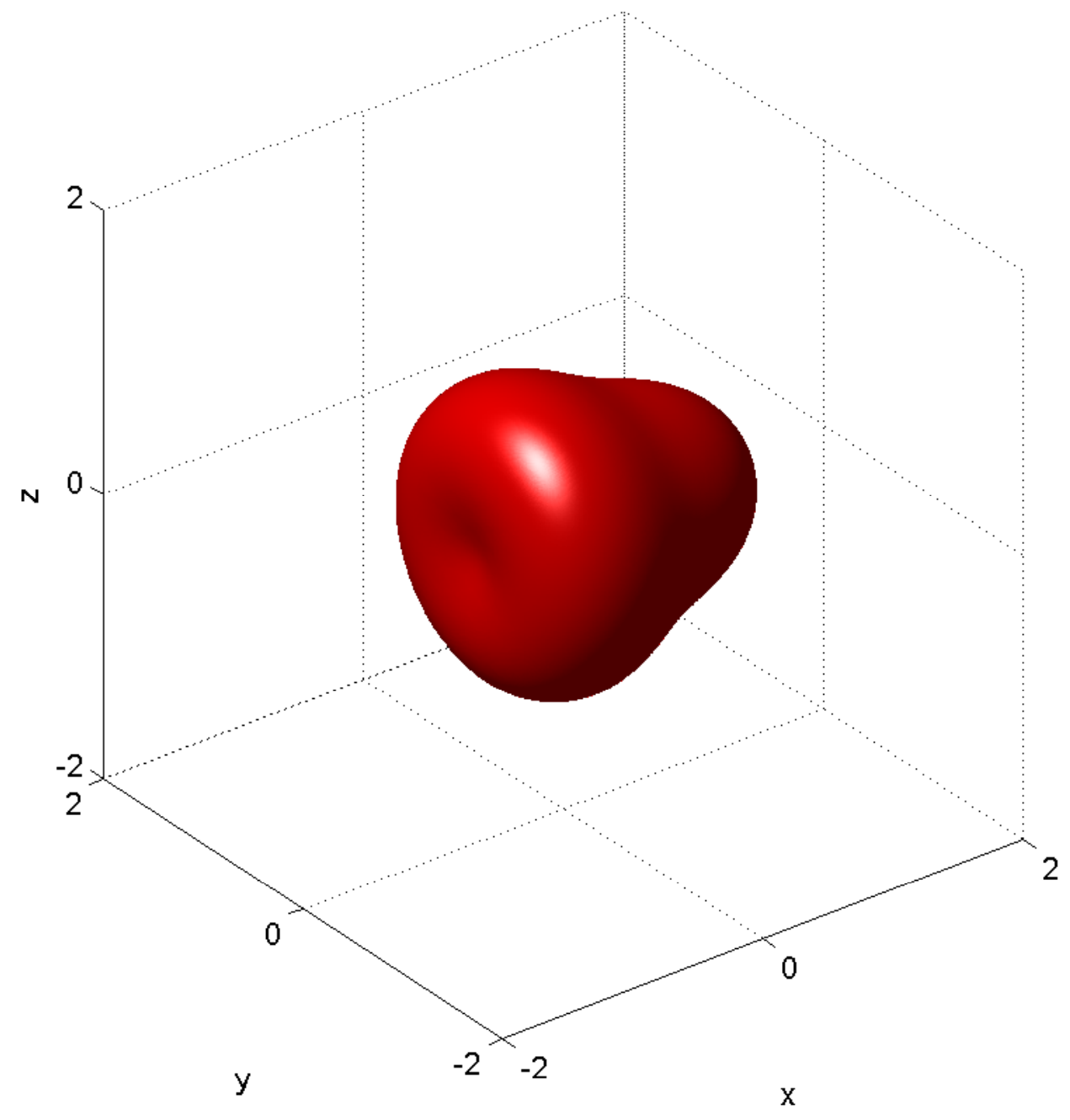}\hfill{}\includegraphics[width=0.24\textwidth]{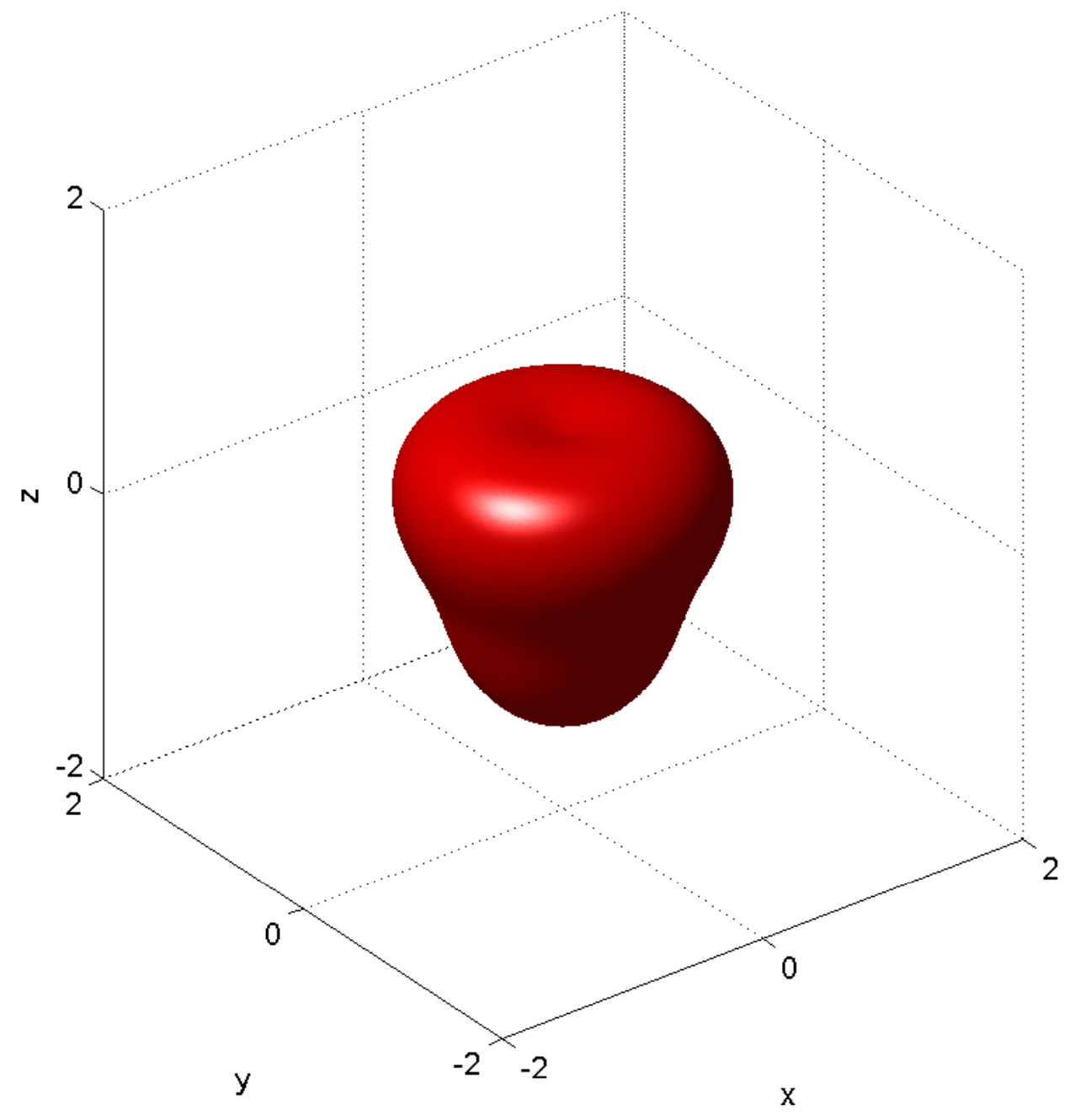}\hfill{}\includegraphics[width=0.24\textwidth]{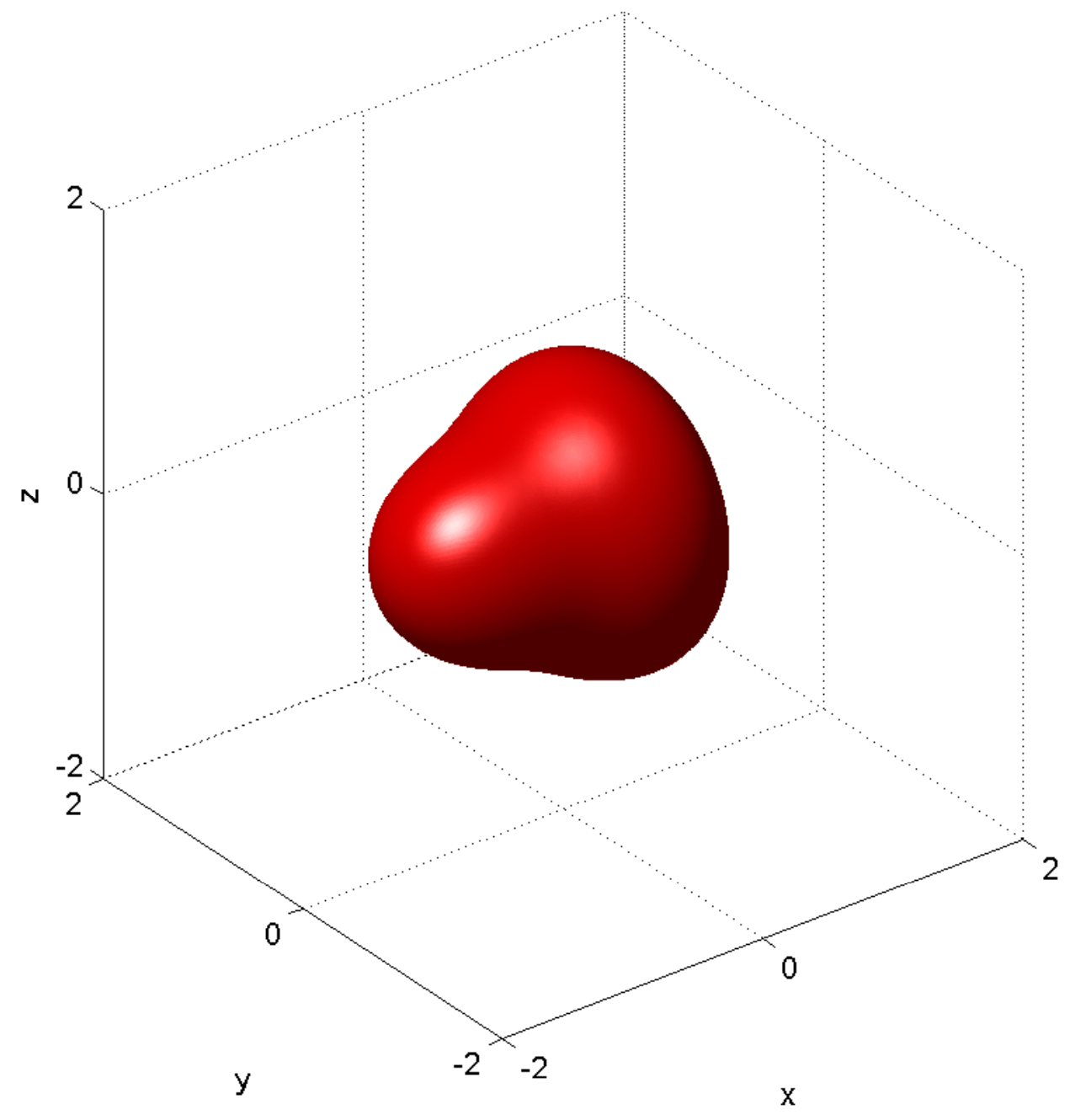}\hfill{}\includegraphics[width=0.24\textwidth]{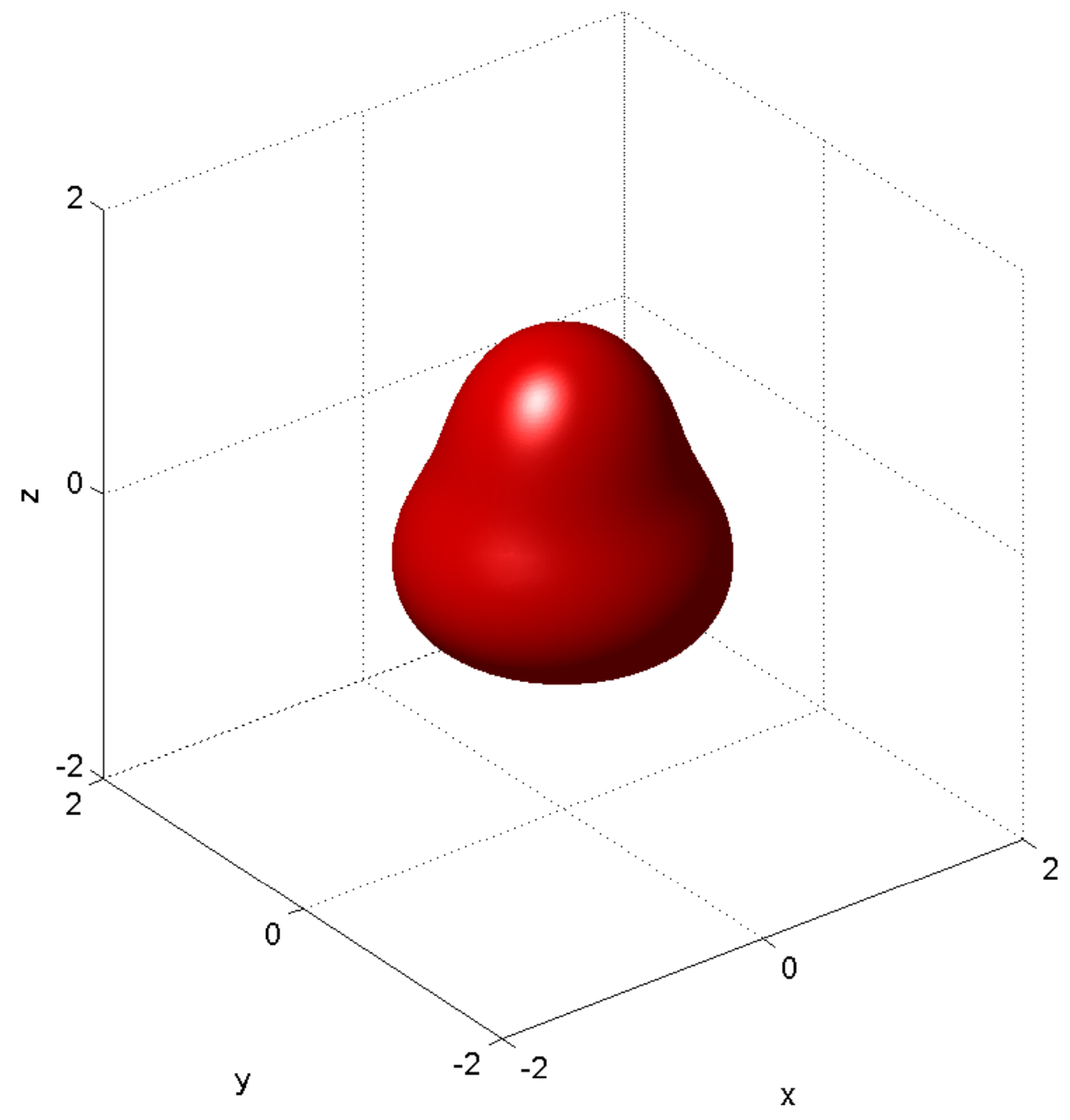}\hfill{}

\hfill{}(a)\hfill{}\hfill{}(b)\hfill{}\hfill{}(c)\hfill{}\hfill{}(d)\hfill{}

\caption{\label{fig:Ex3_basis} Example 2: Basic scatterer components : a reference acorn
with four orientations.}
\end{figure}

The indicator function $W_3(z)$ is adopted to locate regular-size scatterer components.  By the increasing magnitude of the far field patterns, the UFO reference data is first employed for locating purpose.  Fig.~\ref{fig:Ex3_results}(a) tells us that the first unknown component is a UFO and its position is highlighted. What's interesting in Fig.~\ref{fig:Ex3_results}(a) is that it also indicates a ghost highlight which is close to the position of the acorn, which is due to the similarity between the UFO and acorn geometries.  In the next stage, by subtracting the UFO contribution from the total far field data through Step 9 in Scheme R, we try further the far field data associated with the reference acorn geometry and all its possible orientations. For example, associated with the four orientations in  Fig.~\ref{fig:Ex3_basis}, the corresponding reconstruction results are plotted in Figs.~\ref{fig:Ex3_results}(b)-(e). It is found in Fig.~\ref{fig:Ex3_results}(c) the most prominent indicating  behavior which identifies the acorn shape, its location and upside-down configuration of the second unknown component.

\begin{figure}
\hfill{}\includegraphics[width=0.5\textwidth]{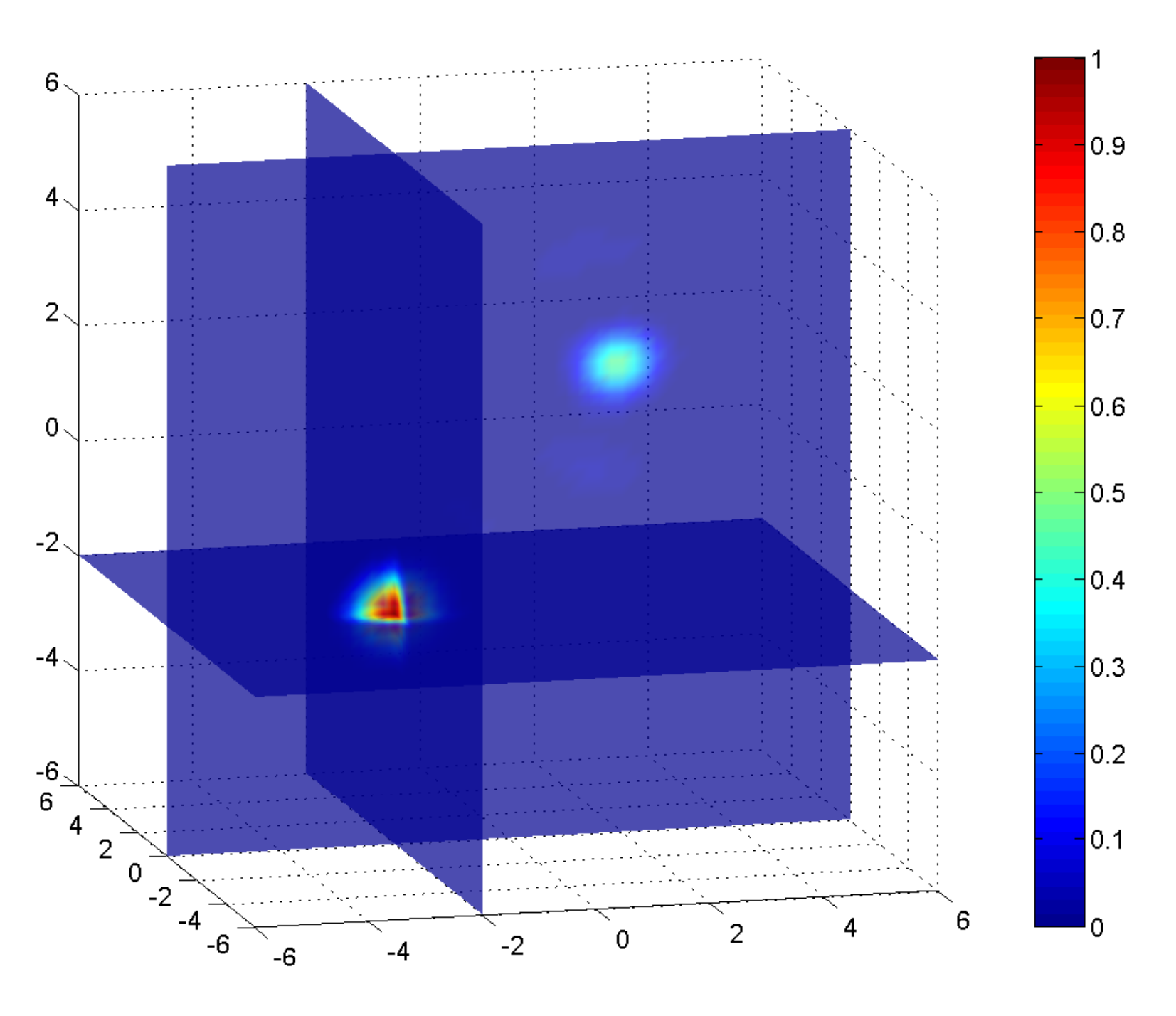}\hfill{}

\hfill{}(a)\hfill{}

\hfill{}\includegraphics[width=0.24\textwidth]{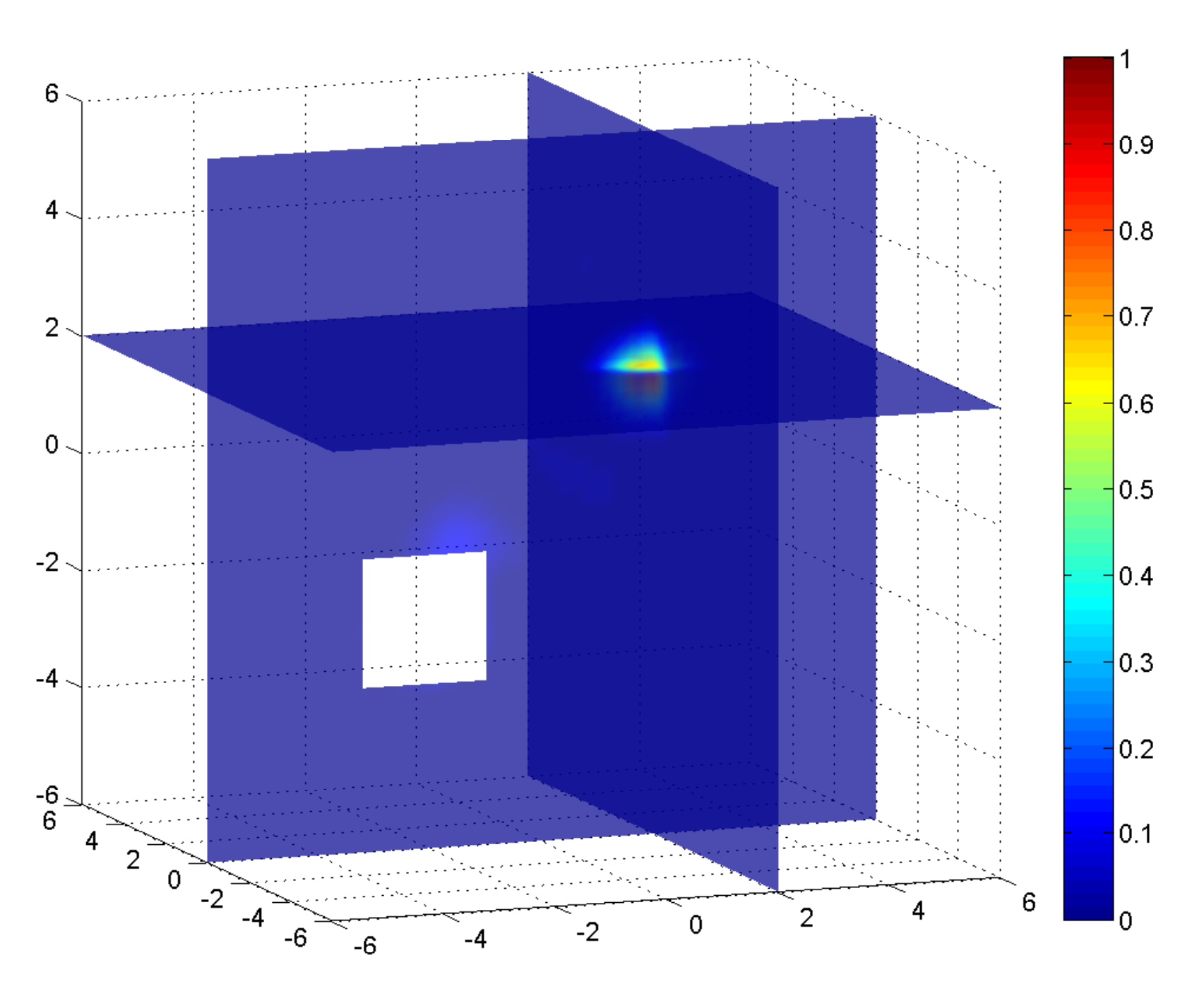}\hfill{}\includegraphics[width=0.24\textwidth]{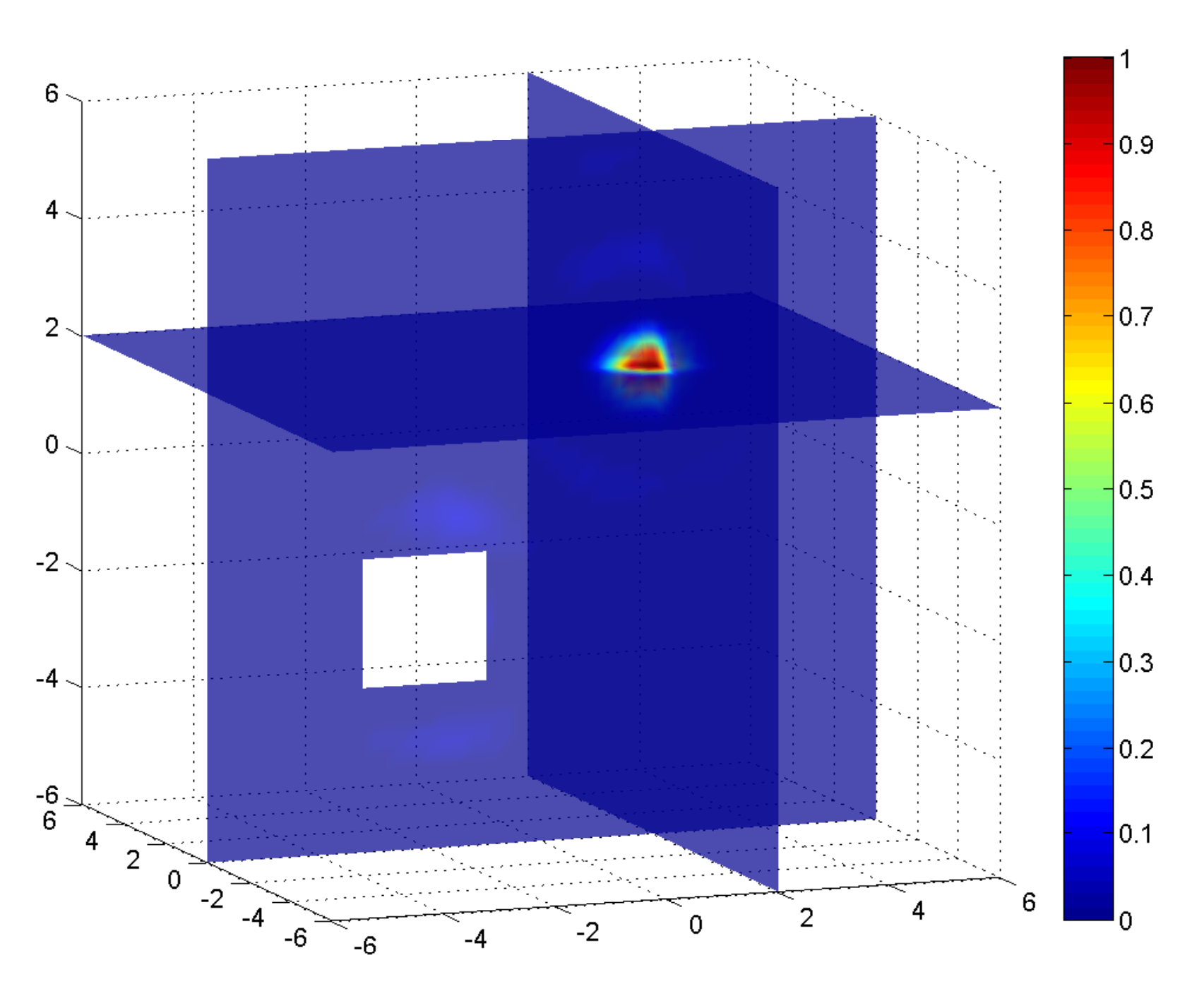}\hfill{}\includegraphics[width=0.24\textwidth]{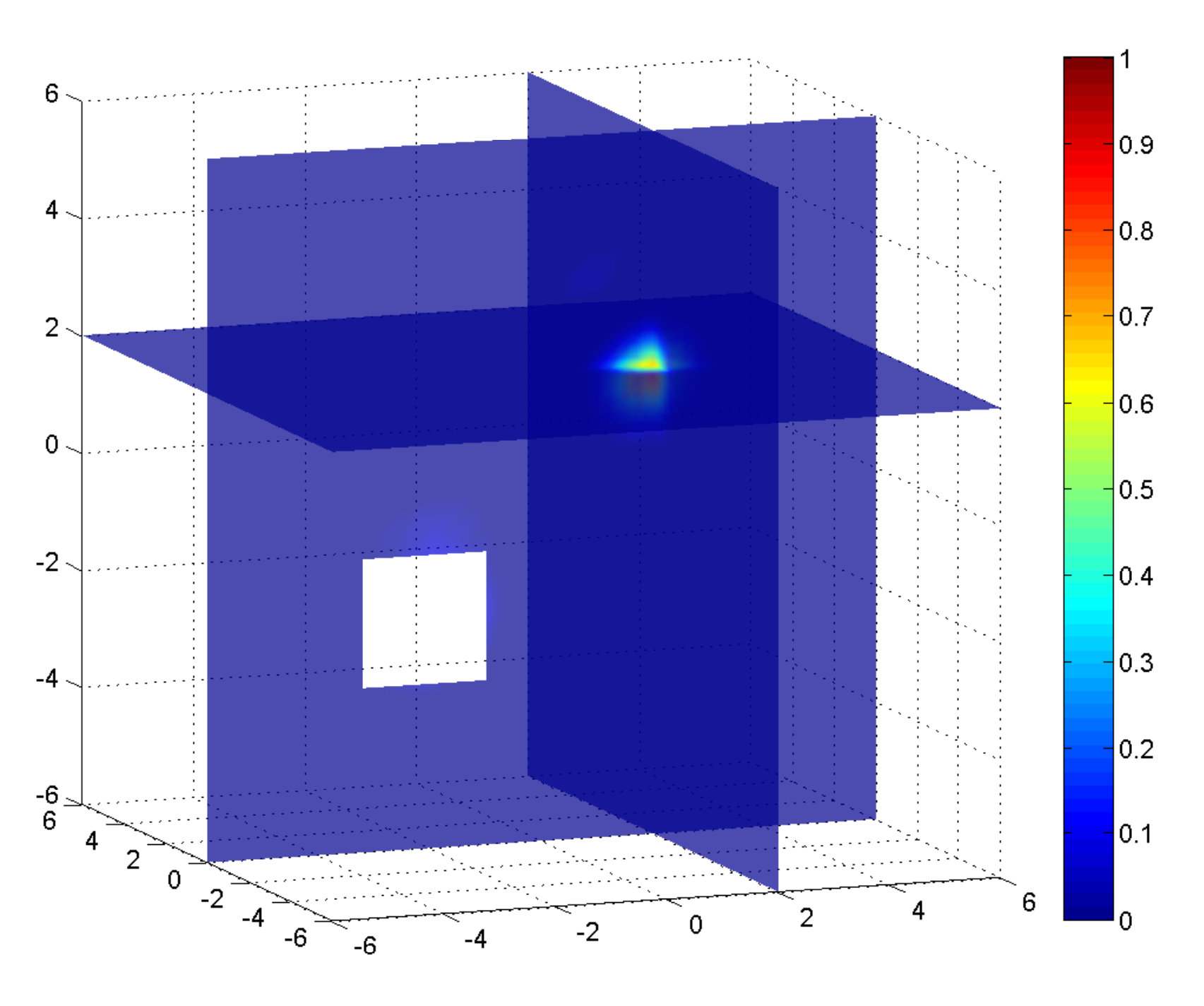}\hfill{}\includegraphics[width=0.24\textwidth]{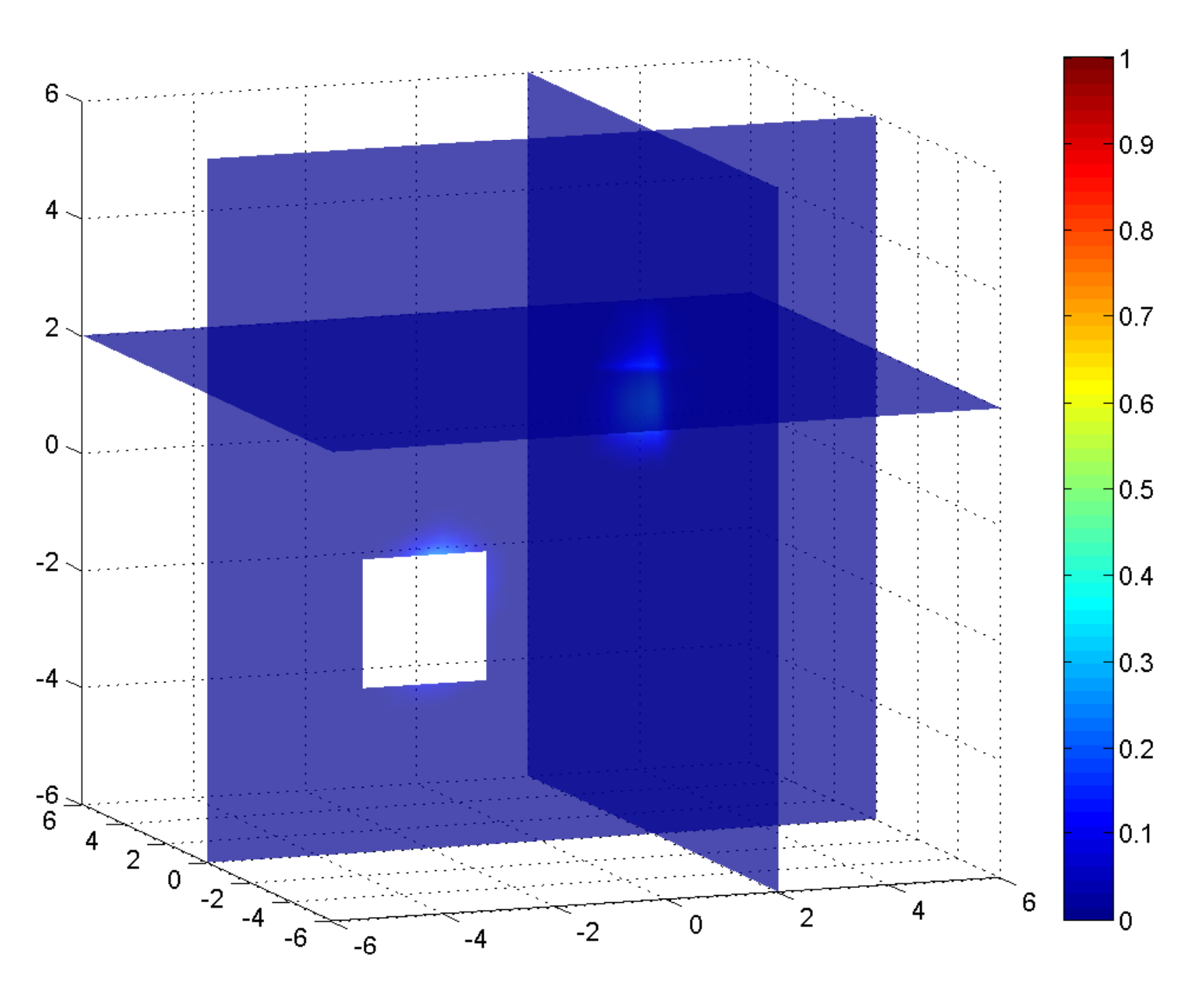}\hfill{}

\hfill{}(b)\hfill{}\hfill{}(c)\hfill{}\hfill{}(d)\hfill{}\hfill{}(e)\hfill{}

\hfill{}basis\_2(pear)\hfill{}\hfill{}\caption{\label{fig:Ex3_results} Example 2. (a) Reconstruction result using the full-wave far field data associated with the reference UFO based on $W_3(z)$; (b)-(e): Reconstruction results using the full-wave far field data associated with the reference acorn and its four orientations based on $W_3(z)$}
\end{figure}

\medskip

\textbf{Example 3} (Multi-scale scatter of multiple components)
In this example, we test further a multi-scale imaging problem using Scheme M.
The true scatterer is composed of a small UFO scaled by one fifth
and an acorn of unitary size. The small UFO is located at $(-2,\,0,\,-2)$, and the
big pear located at $(2,\,0,\,2)$ as shown in Fig. \ref{fig:Ex2_true}.
As for each reference component of \textbf{A} and \textbf{U}, we rotate it every 90 degrees in the $x-y$, $y-z$ and $z-x$ planes.
Three different sizes of the reference components are tested, namely scaled by $0.2$, $1$ and $1.5$.

In the first stage, we extract the information of the regular-size component using the indicator function $W_3(z)$ of Scheme R by computing the inner product with a priori known far-field data associated with those reference scatterer component with different orientations and sizes.  We plot in Fig.~\ref{Ex2_results} the indicator function values of $W_3(z)$ in one-to-one correspondence with the four orientations
of the reference acorn as shown in Fig.~\ref{fig:Ex3_basis}. It can be observed in Fig.~\ref{Ex2_results}(b) that  the highlighted part tells us that the first  regular-size unknown  component is  the approximate location of an  acorn with no scaling and upside-down configuration.  By testing other regular-size components, no significant maxima are found and it is now safe to undergo the second stage for detecting the possible remaining small-size components.

In the next stage,  we adopt the local tuning technique pby performing a local search over a small cubic mesh  around the rough position of the acorn  
determined by the highlighted local maximum in Fig.~\ref{Ex2_results}(b).   In Fig.~\ref{fig:Ex2_resampling}, as the search grid points approach gradually  from $(1.6\,, 0\,,2)$ to $(2\,, 0\,,2)$ (from left to right), the  value distribution of the indicator function in Scheme S displays an gradual change of  the highlighted position. In Fig.~\ref{fig:Ex2_resampling}(c), the red dot  indicates the approximate position of the smaller UFO component, which agrees with the exact one very well. In such a way, the  small UFO component could be correctly identified and positioned, and it helps us fine tune the position of the acorn and update it  to be  $(2\,, 0\,,2)$.

\begin{figure}
\hfill{}\includegraphics[width=0.6\textwidth]{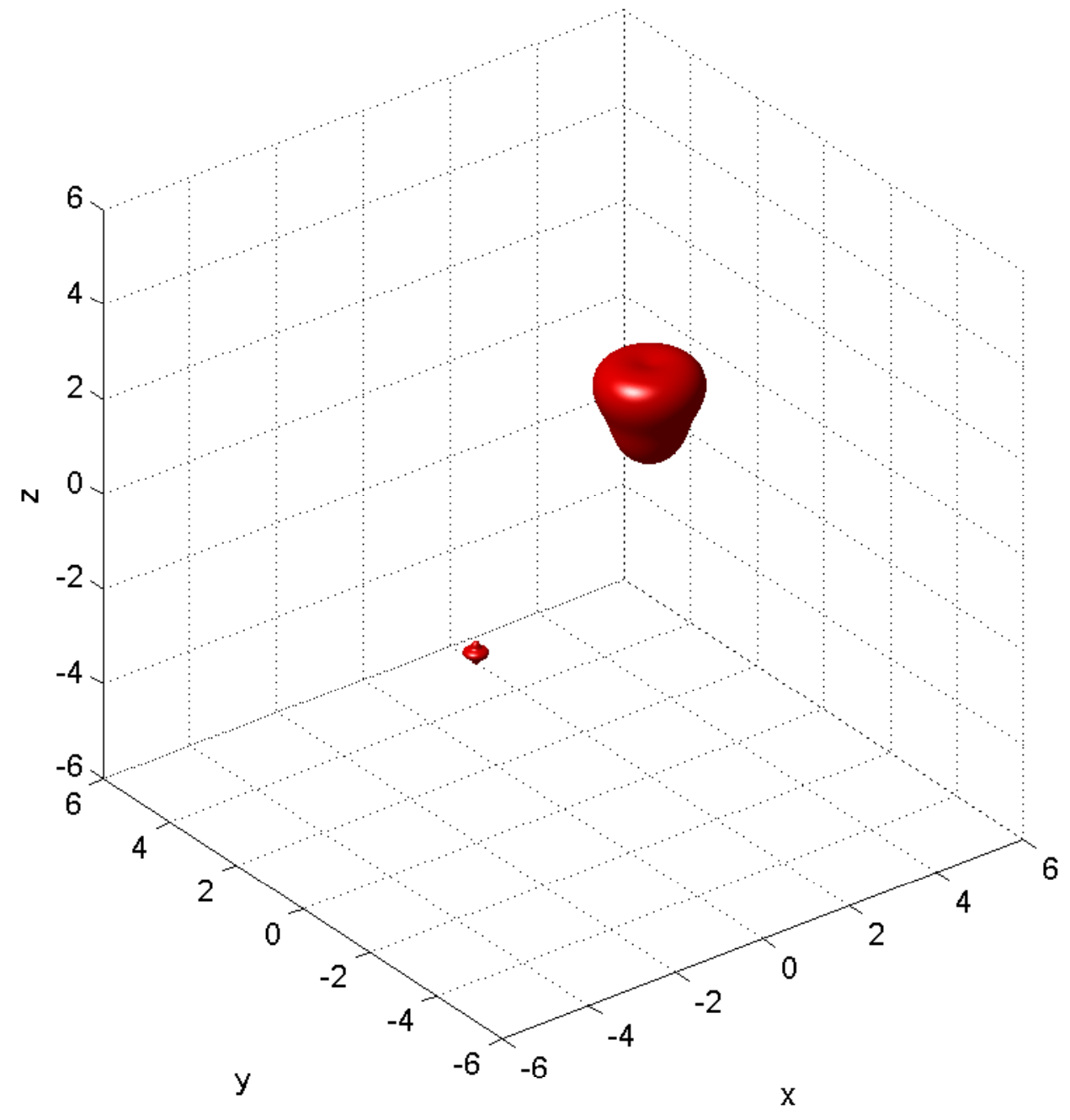}\hfill{}

\caption{\label{fig:Ex2_true}True scatterer in Example 3.}
\end{figure}

\begin{figure}
\hfill{}\includegraphics[width=0.24\textwidth]{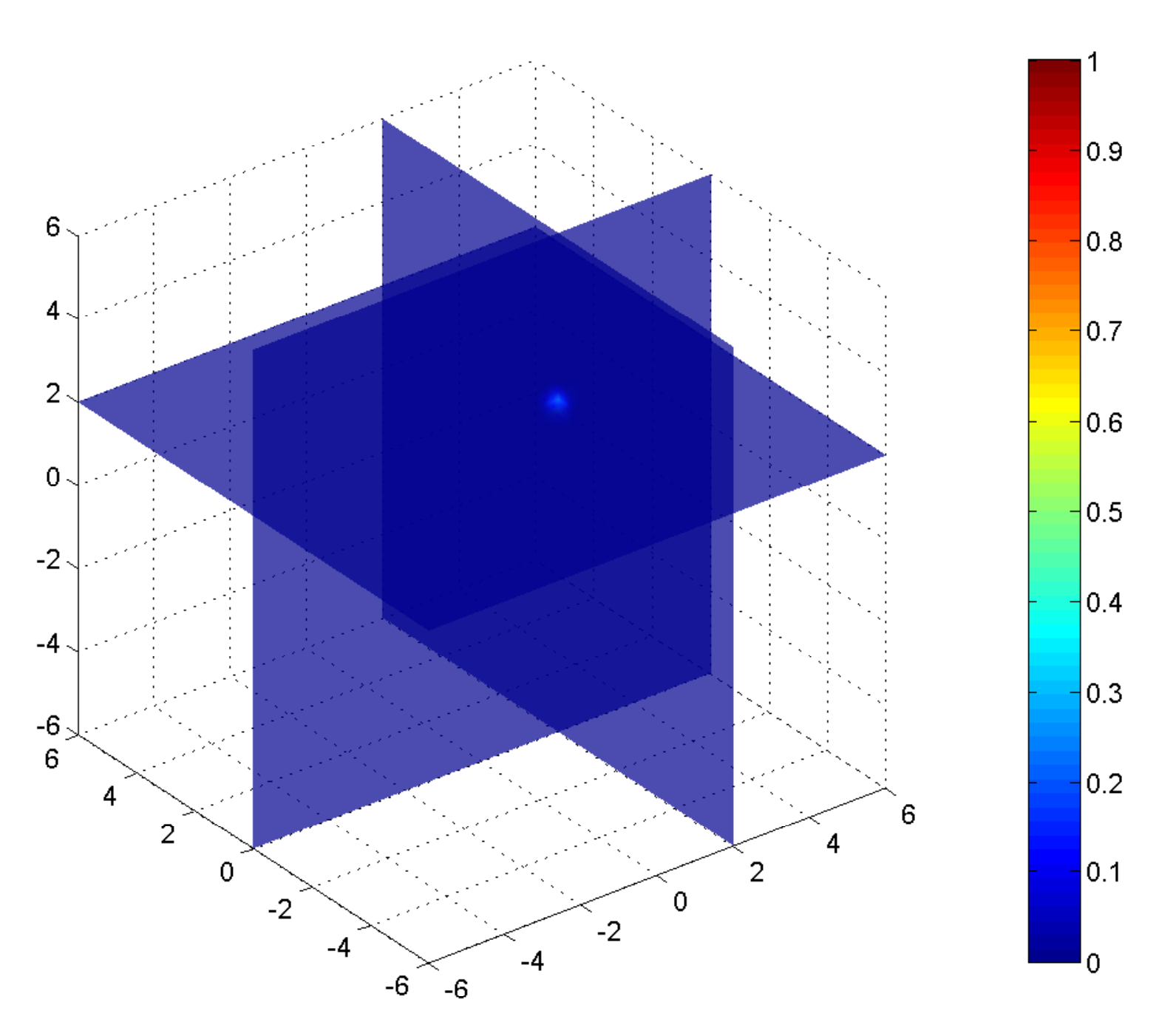}\hfill{}\includegraphics[width=0.24\textwidth]{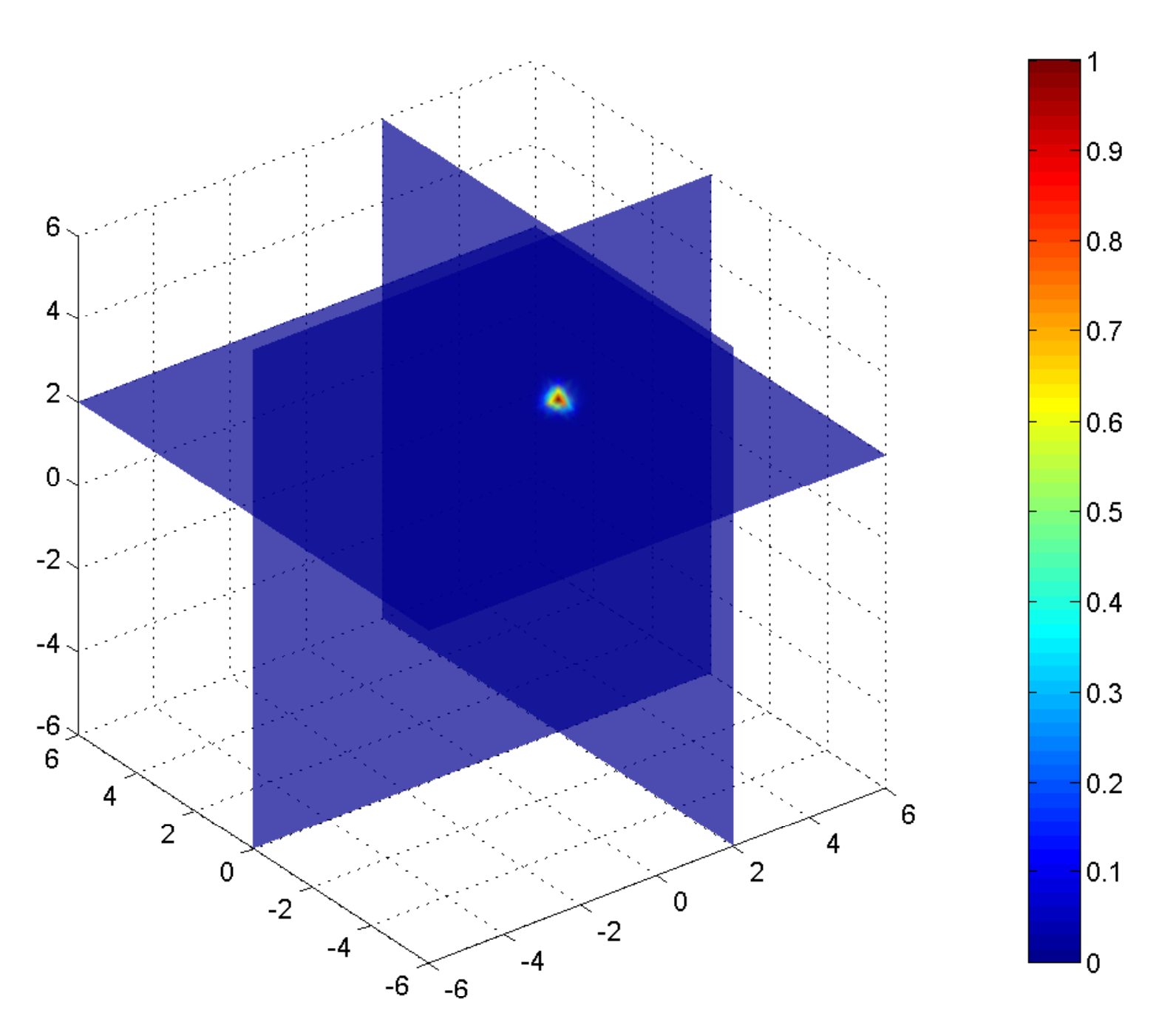}\hfill{}\includegraphics[width=0.24\textwidth]{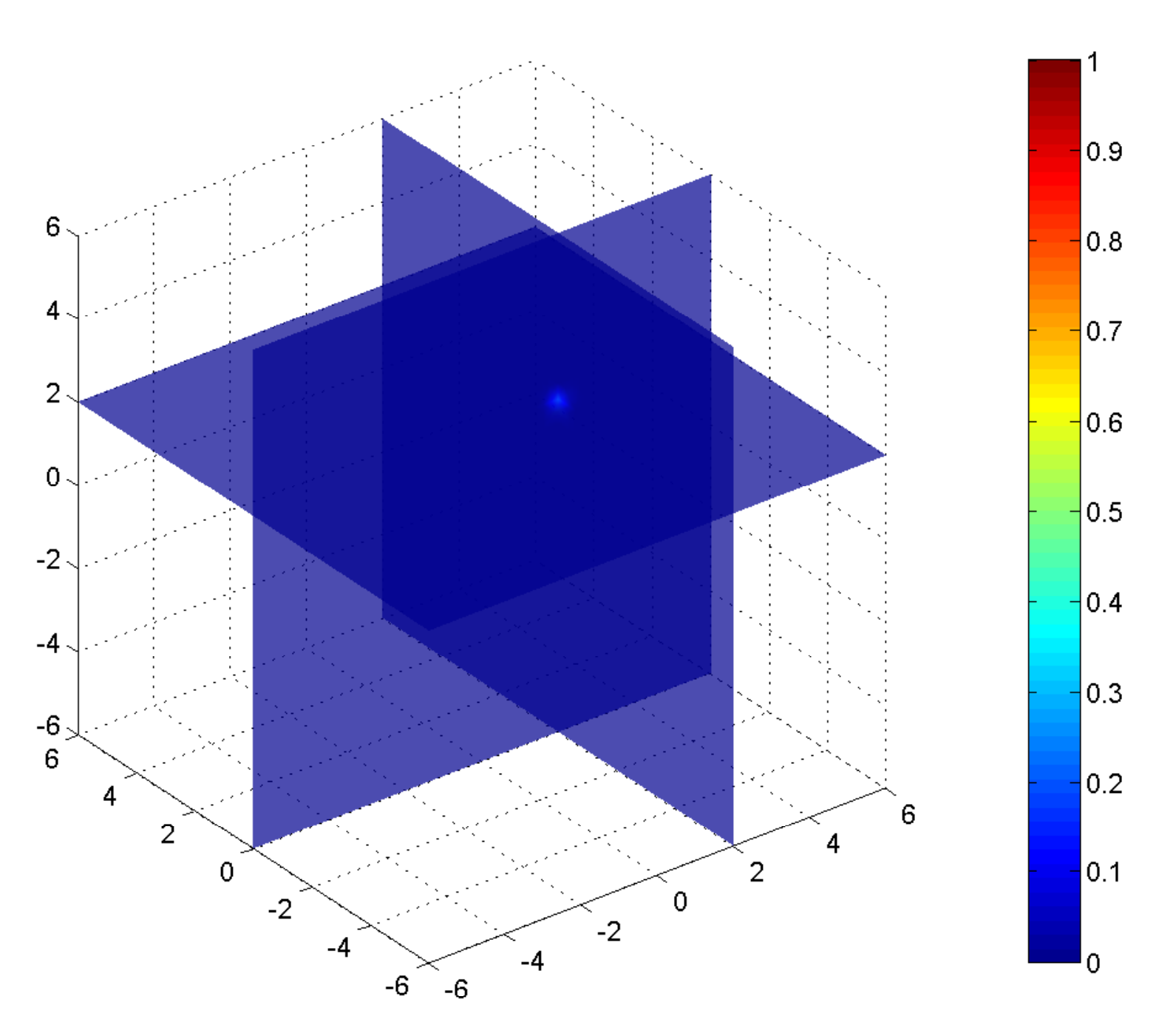}\hfill{}\includegraphics[width=0.24\textwidth]{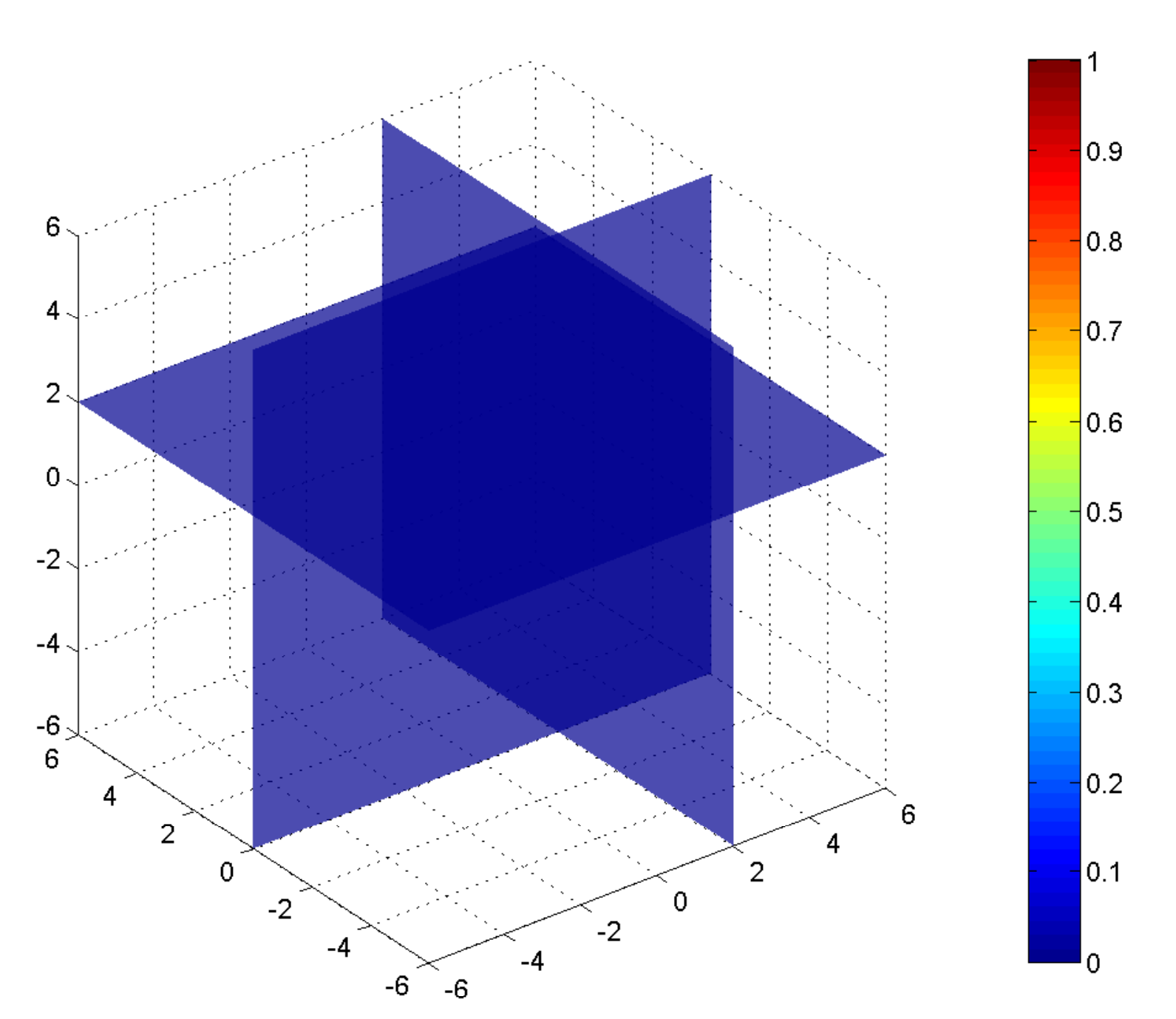}\hfill{}

\hfill{}(a)\hfill{}\hfill{}(b)\hfill{}\hfill{}(c)\hfill{}\hfill{}(d)\hfill{}

\caption{\label{Ex2_results} Reconstruction results in the first stage of Scheme M in Example 3. From left to right: Iindicating plots of $W_3(z)$ by testing with far field data associated with the four orientations in Fig.~\ref{fig:Ex3_basis}. }

\end{figure}

\begin{figure}
\hfill{}\includegraphics[width=0.32\textwidth]{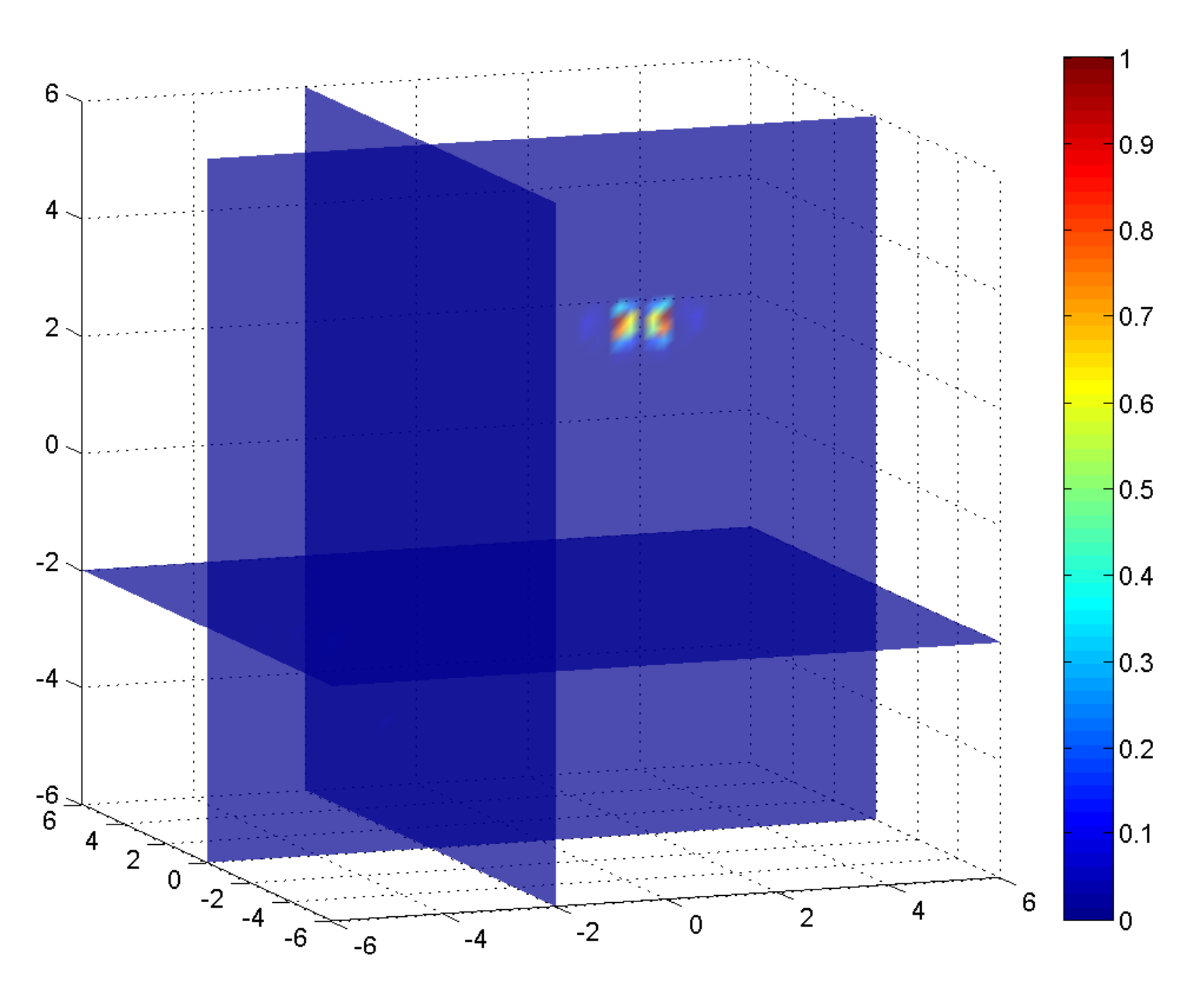}\hfill{}\includegraphics[width=0.32\textwidth]{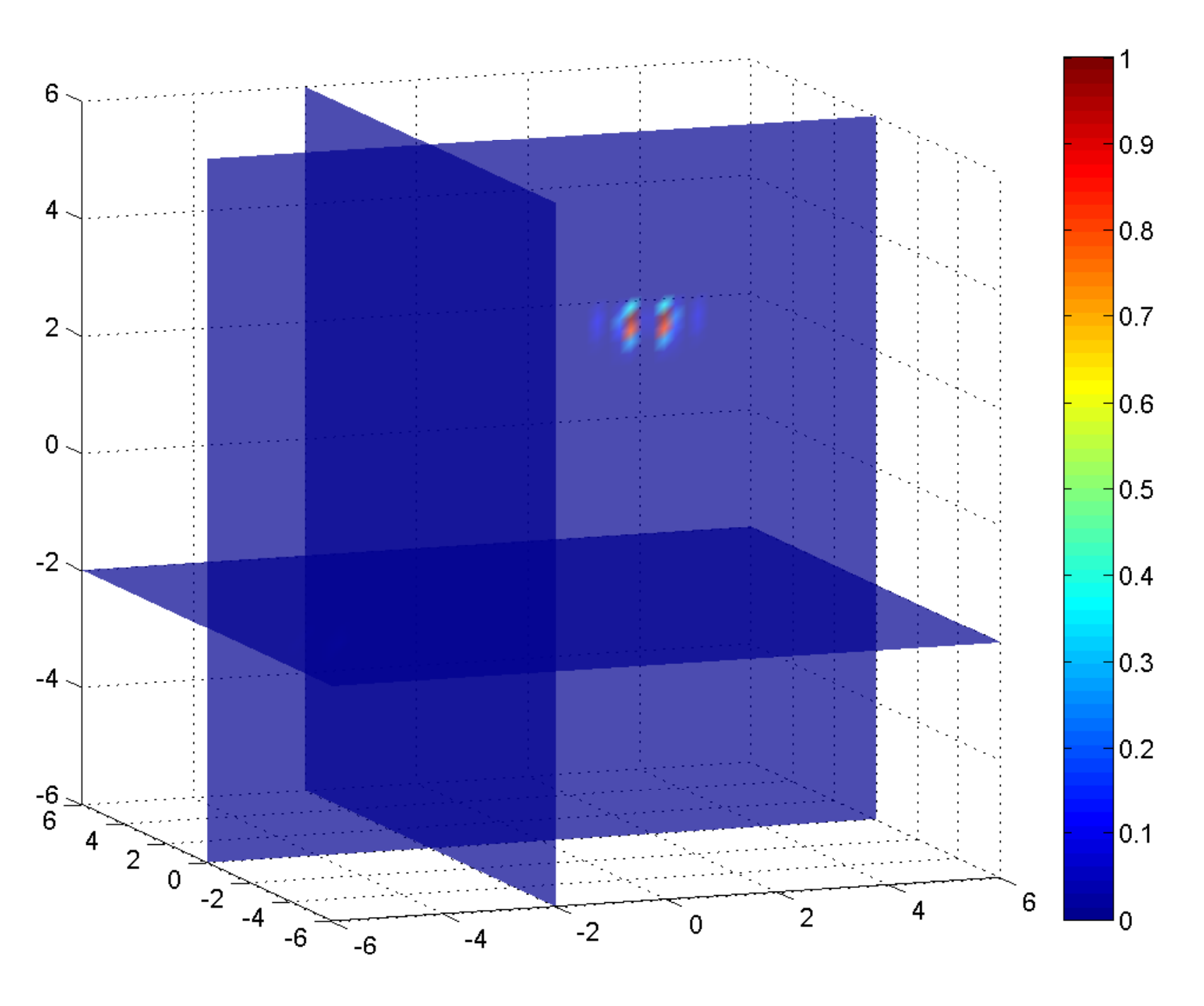}\hfill{}\includegraphics[width=0.32\textwidth]{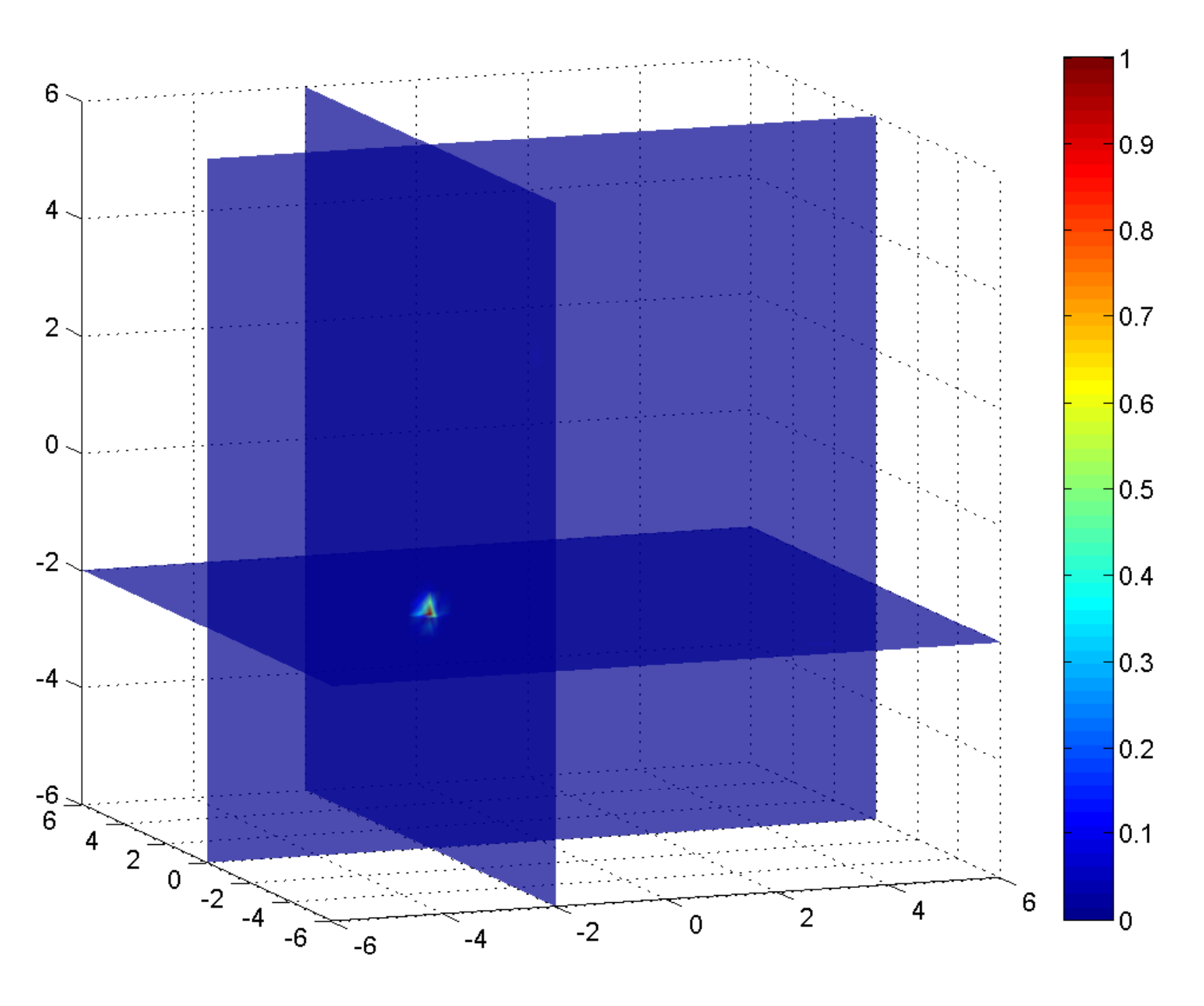}\hfill{}

\hfill{}(1.6,0,2)\hfill{}\hfill{}(1.8,0,2)\hfill{}\hfill{}(2,0,2)\hfill{}

\caption{\label{fig:Ex2_resampling} Reconstruction results by locally tuning the rough location on some typical  local grid points in Example 3. }
\end{figure}

\section{Concluding remarks}
In this work, three imaging schemes S, R and M are developed to identify respectively, multiple small, extended and multiscale rigid elastic scatterers from the far-field pattern corresponding to a single incident plane wave with fixed incident direction and frequency. The incident elastic wave is allowed to be a plane pressure wave, a plane shear wave or a general linear combination of $P$- and $S$- waves taking the form (\ref{plane}).
Relying on the availability of the far-field data, we have developed three indicating functions in each scheme by using the P-part, S-part or the full far-field pattern. Our locating schemes are based on the local \textit{maximum} behaviors of the indicating functions. Rigorous mathematical justifications are provided and
several benchmark examples are presented to illustrate the efficiency of the schemes.

We remark that in Scheme R, if certain a priori information is available
about the possible range of the orientations and sizes of the scatterer components, it is sufficient
for the augmented reference space $\mathscr{A}_h$ in \eqref{eq:a3}  to cover that range only.
In Lemma \ref{Lem:uniqueness}, we have shown uniqueness in locating the position of a translated
elastic body from either the P-part or S-part of the far-field pattern corresponding to a single plane pressure or shear wave. However, we do not know if analogous uniqueness results hold for rotated and scaled elastic bodies; that is, wether or not a single far-field pattern can uniquely determine a rotating or scaling operator acting on the scatterer.

Although only the rigid scatterers are considered in the current study,
the proposed schemes can be generalized to locating multiple multiscale cavities modeled by the traction-free boundary condition on the surface, the Robin-type impenetrable elastic scatterers as well as inhomogeneous penetrable elastic bodies with variable densities and Lam\'{e} coefficients inside. 
To achieve this, one only needs to investigate the analogous asymptotic expansions of the far-field pattern to Lemma \ref{lem2}, which will be used to design the locating functionals for small scatterers. The results in Section \ref{Proof-Th2} remain valid for extended elastic scatterers of different physical natures. Hence, the schemes of locating extended scatterers can be straightforwardly extended to the cases mentioned above.
Our approach can be also extended to the case where only limited-view measurement data are available. Further, the extension to the use of time dependent
measurement data would be nontrivial and poses interesting challenges for
further investigation.

\section*{Acknowledgement}
This work was initiated when G. Hu visited the South University of Science and Technology of China (SUSTC) in June 2013. He gratefully acknowledges SUSTC's hospitality and the financial support from German Research Foundation (DFG) under Grant No. HU 2111/1-1.  The work was supported by the NSF of China (No.\,11371115, 11201453, 91130022) and by the NSF grant, DMS 1207784.

\end{document}